\documentclass[11pt,english]{article}

\usepackage[margin = 2.4cm]{geometry}

\usepackage{amsthm}
\usepackage{amsmath}
\usepackage{amssymb}
\usepackage{setspace}
\usepackage{mathtools}
\usepackage{graphicx}
\usepackage[hidelinks]{hyperref}
\usepackage{thmtools}
\usepackage{thm-restate}
\usepackage{cleveref}

\usepackage{enumerate}
\usepackage{framed}
\usepackage{subcaption}

\usepackage{floatrow}
\usepackage{comment}
\theoremstyle{plain}

\newtheorem*{thm*}{Theorem}
\newtheorem{thm}{Theorem}[section]
\Crefname{thm}{Theorem}{Theorems}

\newtheorem*{lem*}{Lemma}
\newtheorem{lem}[thm]{Lemma}
\Crefname{lem}{Lemma}{Lemmas}

\newtheorem*{claim*}{Claim}
\newtheorem{claim}[thm]{Claim}
\crefname{claim}{Claim}{Claims}
\Crefname{claim}{Claim}{Claims}

\newtheorem{prop}[thm]{Proposition}
\Crefname{prop}{Proposition}{Propositions}

\crefname{cor}{Corollary}{Corollaries}

\crefname{conj}{Conjecture}{Conjectures}

\Crefname{qn}{Question}{Questions}

\newtheorem{obs}[thm]{Observation}
\Crefname{obs}{Observation}{Observations}

\Crefname{ex}{Example}{Examples}

\theoremstyle{definition}

\Crefname{prob}{Problem}{Problems}

\newtheorem{defn}[thm]{Definition}
\Crefname{defn}{Definition}{Definitions}

\theoremstyle{remark}

\renewenvironment{proof}[1][]{\begin{trivlist}
\item[\hspace{\labelsep}{\bf\noindent Proof#1.\/}] }{\qed\end{trivlist}}

\newcommand{\remove}[1]{}

\newcommand{\ceil}[1]{
    \lceil #1 \rceil
}
\newcommand{\floor}[1]{
    \lfloor #1 \rfloor
}

\newcommand{\eps}{\varepsilon}
\newcommand{\whp}{with high probability}

\newcommand{\dirpath}[1]{\overrightarrow{P_{#1}}}

\AtBeginDocument{%
   \def\MR#1{}
}

\DeclareMathOperator{\lf}{lf}
\DeclareMathOperator{\dpl}{L}

\begin{document}


\title{Monochromatic trees in random tournaments}
\date{\vspace{-5ex}}
\author{
	Matija Buci\'c\thanks{
	    Department of Mathematics, 
	    ETH Zurich, Switzerland;
	    e-mail: \texttt{matija.bucic}@\texttt{math.ethz.ch}.
	}
	\and
  Sven Heberle\thanks{
    Department of Mathematics,
    ETH Zurich, Switzerland;
    e-mail: \texttt{heberle.sven}@\texttt{gmail.com}.
  }
  \and
    Shoham Letzter\thanks{
        ETH Institute for Theoretical Studies,
        ETH Zurich, Switzerland;
        e-mail: \texttt{shoham.letzter}@\texttt{math.ethz.ch}.
        Research supported by Dr.\ Max R\"ossler, the Walter Haefner Foundation and the ETH Zurich Foundation.
    }
    \and
Benny Sudakov\thanks{Department of Mathematics, ETH Z\"urich, Switzerland; Email:
\href{mailto:benjamin.sudakov@math.ethz.ch} {\nolinkurl{benjamin.sudakov@math.ethz.ch}}.
Research supported in part by SNSF grant 200021-175573.}
}

\maketitle

\begin{abstract}

    \setlength{\parskip}{\medskipamount}
    \setlength{\parindent}{0pt}
    \noindent

    We prove that, \whp, in every $2$-edge-colouring of the random tournament on $n$ vertices there is a monochromatic copy of every oriented tree of order $O (n / \sqrt{\log n})$. This generalises a result of the first, third and fourth authors who proved the same statement for paths, and is tight up to a constant factor.

\end{abstract}
\section{Introduction} \label{sec:intro}

    Ramsey theory consists of a considerable amount of mathematical results, which, roughly speaking, say that there is no completely chaotic structure, i.e.\ any sufficiently large structure is guaranteed to have a large well-organised substructure. For instance, the famous theorem of Ramsey \cite{ramsey1929problem} states that for any fixed graph $H$, every $2$-edge-colouring of a sufficiently large complete graph contains a monochromatic copy of $H$. The smallest order of a complete graph with this property is called the \emph{Ramsey number of $H$}.
    
    In this paper we study an analogous phenomenon for oriented graphs. An \emph{oriented graph} is a directed graph $G$ obtained by orienting the edges of a simple undirected graph, which is called the \emph{underlying graph} of $G$.
    
    A \emph{tournament} is an oriented graph whose underlying graph is complete. Given oriented graphs $G,H,K$ we write $G\to (H,K)$ whenever in every $2$-colouring of the edges of $G$ there is a blue copy of $H$ or a red copy of $K$. In the special case that $H=K$, we write $G \to H$. The \emph{oriented Ramsey number} of $H$ is defined to be the smallest $N$ for which every tournament $G$ on $N$ vertices satisfies $G \to H$.

    Note that unlike the standard Ramsey numbers which are always finite, in the oriented setting if $H$ contains a directed cycle then its oriented Ramsey number may be infinite. To see this consider the following colouring: fix an ordering of the vertices and colour all forward edges blue and all backward edges red. This $2$-coloured tournament does not contain any monochromatic directed cycles. In particular, it does not have a monochromatic copy of $H$ if $H$ contains a directed cycle. Moreover, it is easy to see that every $2$-edge-coloured tournament on $N$ vertices contains a monochromatic transitive tournament on roughly $\log_4 N$ vertices, from which it follows that it contains a monochromatic copy of every acyclic graph on at most $\log_4 N$ vertices. Hence, the oriented Ramsey number is finite if and only if $H$ is acyclic. 
    
    Let us start by investigating Ramsey numbers of directed paths. Denote the directed path on $n$ vertices by $\dirpath{n}$, where by a directed path we mean an oriented graph obtained from a path by orienting all its edges in the same direction. The celebrated Gallai-Hasse-Roy-Vitaver theorem \cite{gallai1968directed,hasse1965algebraischen,roy1967nombre,vitaver1962determination} says that any directed graph, whose underlying graph has chromatic number at least $n$, contains a $\dirpath{n}$ as a subgraph. It follows that $G \to \dirpath{n}$ for every tournament $G$ of order at least $(n-1)^2+1$; indeed, given a red and blue colouring of $G$, either the graph of red edges or the graph of blue edges has chromatic number at least $n$, implying the existence of a monochromatic path on at least $n$ vertices. This statement is sharp. To see this, consider a \textit{transitive tournament} on $(n-1)^2$ vertices. We partition the vertices into sets $A_i,$ each of size $n-1$, while preserving the order. We colour all edges inside some set $A_i$ blue, and all other edges red. It is easy to see that there is no monochromatic path on $n$ vertices in this colouring. This shows that the oriented Ramsey number of $\dirpath{n}$ is $(n-1)^2+1$.

    It is interesting to consider oriented Ramsey numbers of further acyclic graphs, and the natural next example are trees. It turns out that oriented trees behave similarly to paths in terms of their oriented Ramsey numbers: it was proved by Buci\'c, Letzter and Sudakov \cite{complete-directed-ramsey} that given any (oriented) tree $T$ on $n$ vertices and any tournament $G$ on $cn^2$ vertices (where $c$ is a positive constant), we have $G \to T$, i.e.\ the oriented Ramsey number of any tree of order $n$ is at most $cn^2$. 
    
    This result resolves (up to a constant factor) the question of, given $n$, finding the smallest $N$ such that every $2$-colouring of every tournament of order $N$ is guaranteed to have a monochromatic copy of $T$ for any tree $T$ of order at most $n$. However, intuitively it seems that examples of tournaments for which the bound is tight are close to being transitive. Therefore, it is natural to ask whether in tournaments that are `far from being transitive' larger monochromatic trees are guaranteed; this question was asked implicitly, for paths, by Ben-Eliezer, Krivelevich and Sudakov \cite{ben2012size}. A natural candidate for such a tournament is the \emph{random tournament}, in which the orientation of each edge is chosen independently and uniformly at random. They showed that, with high probability, every $2$-colouring of a random tournament on $N$ vertices contains a monochromatic directed path of length at least $\frac{cN}{\log N}$. They also showed that every tournament of order $N$ can be $2$-coloured without creating monochromatic paths of length $\frac{3N}{\sqrt{\log N}}$, using the following $2$-colouring of a given tournament $G$ of order $N$. It is well known and easy to see that any tournament of order $N$ has a transitive subtournament of order $\log N$. Using this we can partition the vertices of $G$ into transitive subtournaments $A_i$ of order $\frac{\log N}{2}$ and a remainder $A_0$ of at most $\sqrt{N}$ vertices. We now $2$-colour each of $A_i$, as described above, to ensure that the longest monochromatic path within $A_i$ is of length $\sqrt{|A_i|}$, and we colour the edges in $A_0$ arbitrarily. We then colour all edges from $A_i$ to $A_j$ blue if $i<j$ and red if $i > j$. In this colouring, the longest monochromatic path has length at most $\frac{2N}{\log N} \sqrt{ \frac{\log N}{2}  }+\sqrt{N} \le \frac{3N}{\sqrt{\log N}}.$
    
    In a later paper Buci\'c, Letzter and Sudakov \cite{path-vs-path} showed that, \whp, any $2$-colouring of a random tournament on $N$ vertices contains a monochromatic directed path of order at most $\frac{cN}{\sqrt{\log N}}$, which is tight up to a constant factor, due to the above upper bound from \cite{ben2012size}. They also showed that the same result holds for \textit{oriented paths}, which are paths in which edges are not required to follow the same direction. Following up in this direction, they asked whether the same holds for general oriented trees. The main result of this paper answers this question in the affirmative.
    
    \begin{restatable}{thm}{main} \label{thm:main}
        There is a constant $c > 0$ such that, \whp{}, a random tournament $G$ on $N$ vertices satisfies $G \to T$, where $T$ is any oriented tree on at most $\frac{cN}{\sqrt{\log N}}$ vertices.
    \end{restatable}
    
    Note that unlike for the standard Ramsey numbers, where the ground graph is complete on $N$ vertices, the oriented Ramsey numbers allow any tournament on $N$ vertices as a ground graph. This suggests that taking the ground graph to be the complete directed graph is perhaps a more natural directed analogue of the standard Ramsey theory, where the \emph{complete directed graph} on $n$ vertices, denoted by $\overleftrightarrow K_N$, is the graph in which between any two vertices $i\neq j$ both directed edges, $ij$ and $ji$, are present. Harary and Hell \cite{harary74} and Bermond \cite{bermond1974some} introduced the notion of the \emph{directed Ramsey number} of an oriented graph $H$, which is defined to be the least $N$ such that every $2$-edge-colouring of $\overleftrightarrow K_N$ contains a monochromatic copy of $H$. The directed Ramsey numbers of directed paths were determined by Gy\'arf\'as and Lehel \cite{gyarfas1973ramsey}, based on a result of Raynaud \cite{raynaud1973circuit}, and, independently, by Williamson \cite{williamson73}. Buci\'c, Letzter and Sudakov \cite{complete-directed-ramsey} generalised these results to oriented trees, and also to the $r$-coloured variant. This result plays a role in our argument of the proof for \Cref{thm:main}. 
    
    Note that while the problem for random tournaments is seemingly more similar to the oriented Ramsey numbers, as the base graphs in both cases are tournaments, it turns out the directed Ramsey numbers are more relevant for our arguments. The main reason is that for random tournaments and complete directed graphs between any two not-too-small sets of vertices $A$ and $B$ there are many edges from $A$ to $B.$ However, because this does not hold for smaller sets (i.e.\ of order at most about $\log N$), the bound for random tournaments is somewhat worse than for complete graphs. Our proof of \Cref{thm:main} relies only on a property of this kind, so the conclusion of Theorem \ref{thm:main} actually holds for any sufficiently pseudorandom tournament; we refer the reader to \Cref{sec:prerequisites} for more details.     
    \subsection{Organisation of the paper}
        In the next section we give an overview of the proof of Theorem \ref{thm:main}. In Section \ref{sec:prerequisites} we introduce some results that we will need throughout the rest of the paper. We then turn to the proof of the asymmetric generalisation of \Cref{thm:main}, which we split into two parts. The first part is presented in Section \ref{sec:path-vs-tree} and deals with the special case when one of the trees is assumed to be a directed path. The second part of the argument, presented in Section \ref{sec:tree-vs-tree}, shows how to use this special case to obtain the general result.
    
        We do not make any effort to optimise the constants presented in this paper. We also neglect rounding whenever it is not relevant for the argument. Given a $2$-colouring of a graph, we call the colours red and blue. When we consider paths and trees we always assume they are oriented, i.e.\ between two vertices there is at most one directed edge. Logarithms are always taken in base $2$, unless stated otherwise.

\section{Overview} \label{sec:overview}

    In this section we give an overview of our arguments. Our aim is to prove that given $n$ and $m$, a random tournament $G$ on $N$ vertices satisfies $G \to (T, S)$ for every oriented trees $T$ and $S$ of order $n$ and $m$, respectively, where $N \ge c (n + m + \sqrt{nm \log(n+m)})$ and $c$ is an absolute positive constant. Our proof is divided into two main parts: in the first, we prove it under the assumption that one of $T$ and $S$ is a directed path, and in the second we deduce the general result. In the remainder of this section, we outline the arguments we use in each of these cases.
    
    \subsubsection*{Tree vs.\ path}
        This is the longest and a more difficult part of the proof, here $T$ is assumed to be a \textit{directed tree} (i.e.\ its edges are directed from a root or vice versa) on $m$ vertices, and $S$ is a directed path $\dirpath{n}$. We first prove the desired result under the assumption that $T$ has not-too-many leaves (namely, at most $m^{1/6}$ leaves). Our aim is to find a red copy of $T$ or a blue copy of $\dirpath{n}$.
    
        We distinguish three types of cycles: long cycles (length at least $bm^{1/3}$), short ones (length at most $am^{1/3}$), and medium ones (all remaining cycles). We now consider two cases: when there exists many pairwise vertex-disjoint  medium or long blue cycles, or when there is a large set spanning no medium or long blue cycles; it is easy to see that one of these cases holds.
    
        \paragraph{Case 1. many disjoint medium or long blue cycles.}
    
            In this case we aim to find a specific structure, which we call \emph{red-blue pairs}. This structure consists of many pairwise disjoint sets, $A_1, B_1, \ldots, A_t, B_t$, of suitable size, such that all $A_i - B_i$ edges are red and each set $A_i$ is contained in a blue path $P_i$, where the $P_i$'s are pairwise vertex-disjoint (see Figure \ref{fig:red-blue-pairs}).
            
            We show how to use this structure to find the red tree or the blue path of desired length. To this end, we construct a $2$-edge-coloured auxiliary complete directed graph, where the edge $ij$ is coloured blue if there are many blue edges going from $A_i$ to $A_j$ in $G$ and red otherwise. Applying the directed Ramsey result for trees from (see Theorem \ref{thm:complete-ramsey}) to this auxiliary graph we find a long blue path or a certain carefully chosen red tree (this is obtained from a suitable split of the tree into smaller subtrees which we call a \emph{tree-split}; see Subsection \ref{subsec:tree-splits} and Figure \ref{fig:tree-split}).
    
            If we find a blue path, we lift it to a blue $\dirpath{n}$ in $G$ making use of the blue paths $P_i$ from our structure. If, instead, we find the red tree, we make use of the red bipartite graphs $G[A_i,B_i]$ to embed a subtree of $T$ within it and connect these embeddings in an appropriate fashion to obtain the full $T$.
    
            Finally, we explain how to find red-blue pairs by exploiting assumptions on the blue cycle structure in each of the following two cases.
            \begin{itemize}
                \item[(1a)]
                    Many disjoint medium blue cycles.
                    
                    We define an auxiliary $2$-coloured complete directed graph $H$, whose vertices are medium blue cycles, and for cycles $C_1$ and $C_2$, edge $C_1 C_2$ is blue if a constant fraction of the vertices in $C_1$ have a blue out-neighbour in $C_2$, and otherwise the edge is red. 
                    
                    It is easy to see that there is either a large red-red matching in $H$, which translates into the desired red-blue pairs structure; or there is a long blue path, which translates into a blue $\dirpath{n}$ in the original tournament.
                \item[(1b)]
                    Many disjoint long blue cycles that span no medium blue cycles.
                    
                    In this case we observe that we can find many disjoint blue cycles with no long blue chords. This allows us to obtain a red-blue pairs structure, by letting the sets $A_i$ and $B_i$ be intervals of the long blue cycles.
            \end{itemize}
    
        \paragraph{Case 2. a large set of vertices spanning no medium or long blue cycle.}
    
            We first show that, in this case, there exist many pairwise disjoint sets $U_1, \ldots, U_{\ell}$ of suitable size such that very few of the edges from $U_i$ to $U_j$, with $i < j$, are blue. Using the version of Theorem \ref{thm:main} for paths, which was proved in \cite{path-vs-path}, each set $U_i$ contains many pairs of vertices joined by a long blue path in $U_i$, or many pairs joined by a long red path; in the former case we say that the set $U_i$ is blue, and otherwise we say that it is red. We now consider two cases.
    
            \begin{itemize}
                \item[(2a)]
                    Most of the sets are red.
                    
                    In this case we consider a split of the tree $T$ into subpaths (in Subsection \ref{subsec:tree-splits} we show how to obtain such a \emph{path-split}). We embed each subpath within a specific $U_i$, where we exploit the fact that we have many options for both start and end vertex of the subpath and the fact that most of the forward edges between the $U_i$'s are red, to embed and connect the paths and obtain a red $T$.
                \item[(2b)]
                    Most of the sets are blue.
                    
                    We define an auxiliary $2$-coloured complete directed graph $K$ whose vertices are the blue $U_i$'s, and an edge $U_iU_j$ is coloured blue if $i>j$ and if there is a blue edge from every large subset of $U_i$ to every large subset of $U_j$, and red otherwise.  
                    
                    As before, we note that $K$ contains either a large red-red matching, or a long blue directed path. In the latter case we lift the path to a blue $\dirpath{n}$ in $G$. If the former holds, we find many large bipartite graphs, corresponding to edges of the matching, such that almost all of their edges are red. We use these graphs and the fact that almost all forward edges between sets $U_i$ are red to embed a red $T$, similarly to the first case.
            \end{itemize}
    
    \paragraph{Removing the restriction on the number of leaves.}
     Throughout Subsection \ref{ssec:few-leaves} we were assuming that $T$ has at most $m^{1/6}$ leaves, which was necessary in order to control the number of subtrees we obtain in various splits of $T$. In Subsection \ref{ssec:general-tree-vs-path} we show how to remove this assumption. For this we introduce another kind of split of $T$ which we call the \textit{core-split} (see Subsection \ref{subsec:tree-splits} for details) which splits $T$ into not too many subtrees, each of which has at most $m^{1/6}$ leaves. Assuming there is no blue $\dirpath{n}$ we find a short sequence of large sets such that each has a large number of red out-neighbours in the next set of the sequence. This we can do because otherwise we show there is a set which has a lot of blue edges which allow us to find the blue $\dirpath{n}$. Finally, we iteratively find parts of the core-split (or find a blue $\dirpath{n}$) within these sets using the result from the previous subsection, where we use the large red out-degree towards the next set to ensure we can join all the pieces into a red copy of $T$.

    \subsubsection*{Tree vs.\ Tree}
    
        The rest of the argument consists of three intermediate steps, which generalise the result obtained in the previous section, with the final goal being a version of Theorem \ref{thm:main} for general trees $T$ and $S$.
        
        \paragraph{Step 1. directed tree vs.\ directed tree with $O(1)$ leaves.}
            Let $T$ be out-directed with $O(1)$ leaves, and let $S$ be a directed tree. We observe that if we remove paths from a directed tree $T$, that start at any leaf and stop right before a branching vertex or the root, then the resulting tree $T'$ has at most half the number of leaves of $T$. We iterate a procedure which reduces the search for a red $T$ or a blue $S$ to a search of red $T'$ or blue $S$, using the previous case of path vs.\ tree. 
            
        \paragraph{Step 2. directed tree vs.\ directed tree.}
            Let $T$ and $S$ be out-directed trees. Our aim is to iterate a procedure that reduces the search of a red $T$ or a blue $S$ to a search for a red $T_1$ or a blue $S_1$, where the order of $T_1$ and $S_1$ is smaller than the order of $T$ and $S$ by at least a constant factor. To that end, we consider the $k$-core of a tree $T$, which is the subtree $T'$ consisting of vertices whose number of descendants is at least $|T|/k$. One can show that $T'$ has at most $k$ leaves and that the trees in the forest $T \setminus V(T')$ have order at most $|T|/k$ (see Definition \ref{def:core}). We make use of the previous step which tells us that we can find a red $T'$ or a blue $S$, if $T'$ is the $k$-core of $T$, where $k = O(1)$. Subsequently, we try to embed the trees in $T \setminus V(T')$ within the correct out-neighbourhoods. If we succeed we found a red $T$, otherwise the tree at which we fail is our $T_1.$ We repeat in blue to obtain $S_1$ and iterate until one of the trees drops to constant size when we once again appeal to the previous result.

        \paragraph{Step 3. tree vs.\ tree.}
            Here we rely on the following idea: if $A$ and $B$ are sets such that every vertex in $A$ has large out-degree in $B$ and the vertices in $B$ have large in-degree in $A$, then given a general tree in $T$, we can aim to embed in-directed subtrees of $T$ in $A$ and out-directed subtrees of $T$ in $B$, using the large degrees between the two sets to connect such subtrees. This idea allows us to go from the previous step, where we search for monochromatic directed trees, to a search for a red directed tree or a blue general tree. We then apply this idea again to obtain the desired result for two general trees.
    
\section{Prerequisites} \label{sec:prerequisites}

    In this section we mention some useful facts which we shall use throughout the proof. First, we introduce the notion of pseudorandomness. Let $G$ be an oriented graph. For two disjoint subsets $A,B$ of the vertices we denote by $e_G(A,B)$ the number of edges directed from $A$ towards $B$; when the graph $G$ is clear from the context, we omit the subscript $G$. For a vertex $v$ we denote the out and in-degree of $v$ by $d^+(v)$ and $d^-(v)$.

    \begin{defn} \label{def:pseudo}
        Let $ 0 < \eps< \frac{1}{2}$ and let $k$ be an integer. An oriented graph $G$ is $(\eps,k)$\emph{-pseudorandom} if for any disjoint sets $A,B\subseteq V(G)$ of size at least $k$ we have $e(A,B) \ge \eps |A||B|$.
    \end{defn}

    It is easy to see, e.g., by Chernoff's inequality, that a random tournament is pseudorandom \whp, as stated in the following lemma (see Lemma 6 in \cite{path-vs-path}). In fact, this is the only property of a random tournament that we shall use in our argument.
    \begin{lem} \label{lem:random-pseudorandom}
        Let $0<\eps<\frac{1}{2}$. There exists a constant $\sigma$ such that a random tournament $T$ is $(\eps,\sigma \log |T|)$-pseudorandom, with high probability.
    \end{lem}

    We shall investigate $2$-colourings of graphs, where the colours are called red and blue. Therefore, we extend the notation related to edges by an index $b$ for blue and $r$ for red edges. E.g., $e_r(A,B)$ denotes the number of red edges going from $A$ to $B$ and similarly $d_b^+(v)$ is the blue out-degree of vertex $v$.
    
    The following lemma gives a lower bound on the number of blue edges in a subset of the vertices which contains no red copy of a particular tree. This will allow us to find large sets where every vertex has many red neighbours.
    
    \begin{lem} \label{lem:tree-by-number-of-edges}
        Let $G$ be an $(\eps, \sigma \log N)$-pseudorandom $2$-coloured tournament on $N$ vertices. Suppose that $U \subseteq V(G)$ has the following properties.
        \begin{enumerate}[(i)]
            \item the induced graph $G[U]$ has at most $\frac{\eps^2}{32}|U|^2$ blue edges and
            \item  $\frac{\eps}{4}|U| \ge \sigma \log N$.
        \end{enumerate}
        Then the graph $G$ contains a red copy of any tree of size $\frac{\eps}{4}|U|$.
    \end{lem}

    \begin{proof}
        Let us consider two sets
        \begin{align*}
            X^{+}=\bigg\{ v \in U \ \Big| \  d_r^{+}(v) < \frac{3\eps}{4} |U|\bigg\} \quad \text{and} \quad
            X^{-}=\bigg\{ v \in U \ \Big| \  d_r^{-}(v) < \frac{3\eps}{4} |U|\bigg\},
        \end{align*}
        where the degrees are with respect to the induced subgraph $G[U]$. We are going to show that both sets have size at most $\frac{\eps}{4} |U|$. If this is the case then the induced graph $G[U\setminus (X^+ \cup X^-)]$ has minimum red in and out-degree at least $\frac{3\eps}{4}|U| -\frac{2\eps}{4} |U|=\frac{\eps}{4} |U|$ and then we can greedily find any red tree of size at most $\frac{\eps}{4} |U|$.

        So let us assume that $|X^+| \ge \frac{\eps}{4} |U|$; the argument for $X^-$ is analogous. Let us pick any $\frac{\eps}{4} |U|$ vertices from $X^+$ and denote this set by $Y$. Then
            $$ e_r(Y,U\setminus Y) < |Y| \cdot \frac{3\eps}{4} |U|=\frac{3\eps^2}{16} |U|^2.$$   
        By the first assumption on $U$, we have 
        \begin{equation} \label{eqn:upper-bd}
            e(Y,U \setminus Y) 
            \le e_r(Y, U \setminus Y) + e_b(Y, U \setminus Y) 
            < \left( \frac{3}{16}+ \frac{1}{32} \right) \eps^2 |U|^2 
            = \frac{7 \eps^2}{32} |U|^2.
        \end{equation}
        However, by pseudorandomness and the lower bound on $|U|$, we have
            $$ e(Y,U\setminus Y) \ge \eps |Y|(|U|-|Y|)\ge \frac{\eps^2}{4}\left(1-\frac{\eps}{4}\right) |U|^2 \ge \frac{7 \eps^2}{32} |U|^2,$$
        where the last inequality follows as $\eps < 1/2$. This is a contradiction to \eqref{eqn:upper-bd}, which implies that $|X^+| < \frac{\eps}{4} |U|$, as required.
    \end{proof}

     A \emph{rooted tree} is a tree with a special vertex which we call the \emph{root}. By removing a vertex $v$ in a rooted tree $T$ we obtain a forest $F$. The \emph{descendants} of $v$ are the vertices of $T$ that are not in the tree in $F$ which contains the root; note that each vertex is a descendant of itself.
   
    \begin{defn} \label{def:core}
        Let $T$ be a rooted tree on $n$ vertices and let $k>1$. The $k$\emph{-core} of $T$ is the subtree of $T$ consisting of vertices that have more than $n/k$ descendants in $T$.
    \end{defn}

    \begin{obs} \label{obs:core}
        Let $T$ be a tree on $n$ vertices and let $T'$ be its $k$-core for $k>1$. Then $T'$ has at most $k$ leaves and every tree of the forest $T\setminus V(T')$ has order at most $n/k$.
    \end{obs}

    \begin{proof}
        Suppose that $T'$ has $k$ non-root leaves (the root of $T'$ is the root of $T$). The sets of descendants in $T$ of each leaf of $T'$ are disjoint and have size greater than $n/k$. This implies that $T$ has order greater than $n$, a contradiction. Therefore, $T'$ has at most $k-1$ non-root leaves, so in total it has at most $k$ leaves. 

        Let $S$ be a tree in the forest $T\setminus V(T')$. Suppose $|S| > n/k$, then the root $v$ of $S$ has more than  $n/k$ descendants, but then $v$ should be in $T'$, a contradiction.
    \end{proof}
    
    The next result makes it possible to bound the number of vertices with degree at least $3$ in the underlying graph; we call such vertices \textit{branching}. Note that a tree is a path if and only if it has no branching vertices. Let $\lf(T)$ be the number of leaves in a tree $T$.

    \begin{lem} \label{lem:branching}
        The number of branching vertices is at most $\lf(T) - 1$.
    \end{lem}

    \begin{proof}
        We argue by induction on the number of leaves $k:= \lf(T)$. If $k \le 2$ the tree is a path and paths do not have any branching vertices.

        For the induction step we assume that the statement holds for all trees with $k-1$ leaves. Let $v$ be any leaf of $T$ and $P$ a path from $v$ to the first vertex adjacent to a branching vertex $w$, which exists as $k\ge 3$. Then $T\setminus V(P)$ is a tree with $k-1$ leaves and by induction has at most $k-2$ branching vertices. It follows that the number of branching vertices in $T$ is at most $k-1$ (as $w$ is the only vertex that is branching in $T$ but need not be branching in $T \setminus V(P)$).
    \end{proof}

    We call an oriented tree $T$ \emph{out-directed}, if there is a vertex $v$, which we call the \emph{root} of $T$, such that all the edges in $T$ are directed away from $v$. Similarly we define an \emph{in-directed} tree to have all edges directed towards $v$. A \emph{directed} tree is an out-directed tree or an in-directed tree.

    \begin{obs} \label{obs:transitive}
        Let $T$ be a directed tree on $n$ vertices. Then it is a subgraph of any transitive tournament $G$ on at least $n$ vertices.
    \end{obs}

    \begin{proof}
        We assume, without loss of generality, that $T$ is out-directed. Let $N=|G|$. Since $G$ is transitive there exists an ordering of the vertices $u_1,u_2,\dots,u_N$, such that all edges are directed towards the higher index. Let $v_1,v_2,\dots,v_n$ be an ordering of the vertices of $T$ obtained by a depth first search algorithm starting at the root $v$ of $T$. Since $T$ is out-directed, the ordering has the property that all edges of $T$ are directed towards a higher index and we can embed $v_i$ in $u_i$ for every $i\in[n]$. 
    \end{proof}

    A leaf of an oriented tree is an \emph{out-leaf} if its out-degree is $0$ and an \emph{in-leaf} if its in-degree is $0$. Note that for an out-directed tree the only in-leaf is the root and all other leaves are out-leaves. In fact every out-leaf is a leaf itself.
    
\section{Tree vs.\ path} \label{sec:path-vs-tree}

    In this section we prove a special case of Theorem \ref{thm:main} for a directed tree vs.\ a directed path; here the random tournament is replaced by a pseudorandom tournament.
    
    \begin{thm} \label{thm:tree-vs-path}
       Given $0<\eps< \frac{1}{2}$ and $\sigma > 0$, there exists a constant $c > 0$ such that the following holds. Let $G$ be a tournament on $N$ vertices which is $(\eps, \sigma \log N)$-pseudorandom. Then $ G \rightarrow (\dirpath{n}, T)$, where $T$ is any directed tree on $m$ vertices, as long as $n$, $m \le N/c$ and $nm \le \frac{N^2}{c^2\log N}$.
    \end{thm}
    
    As an intermediate result we prove it first for trees with relatively few leaves (see Subsection \ref{ssec:few-leaves}), we then prove Theorem \ref{thm:tree-vs-path} in Subsection \ref{ssec:general-tree-vs-path}. Before turning to the proofs, we discuss three types of tree splits, which we shall use in the proofs.
    
    \subsection{Tree splits} \label{subsec:tree-splits}
        Our proofs in this section make use of several tree splits; we present them here.
        
        \subsubsection*{The $(c, \alpha)$-tree-split}
    
            Let $T$ be an out-directed tree and let $T'$ be a subtree of $T$. An \emph{extending-leaf} of $T'$ with respect to $T$ is an out-leaf (i.e.\ a non-root leaf) of $T'$, which is not a leaf of $T$. Whenever $T$ is clear from the context we do not mention it.

            \begin{lem} \label{lem:tree-split}
                Let $c\ge2$ and $0<\alpha \le (2c)^{-1}$. Let $T$ be an out-directed tree on $m$ vertices with at most $m^\alpha$ leaves. Then there is a partition of the vertices into subtrees $T_1,\dots,T_\ell$ such that the following properties hold.
                \begin{enumerate}[(i)]
                    \item \label{itm:tree-split-one-in-edge}
                        For every $i\in[\ell]$ there is at most one in-edge towards a vertex of $T_i$ in $T$, and if present it is towards the root of $T_i$, 
                    \item \label{itm:tree-split-extend-leaf}
                        the only vertices with out-edges leaving $T_i$ are extending-leaves of $T_i$,
                    \item \label{itm:tree-split-odd-level}
                        each extending-leaf of $T_i$ lies in an even level (i.e.\ its distance to the root of $T_i$ is even) and it has out-degree exactly one in $T$,
                    \item \label{itm:tree-split-size}
                        $|T_i| \le 6 m^{c\alpha},$ for all $i \in [\ell],$ and
                    \item \label{itm:tree-split-number}
                        $\ell \le 2 m^{1-c\alpha}$.
                \end{enumerate}
            \end{lem}
           
           \begin{figure}[ht]
                \caption{A tree-split of a tree.}
                \includegraphics{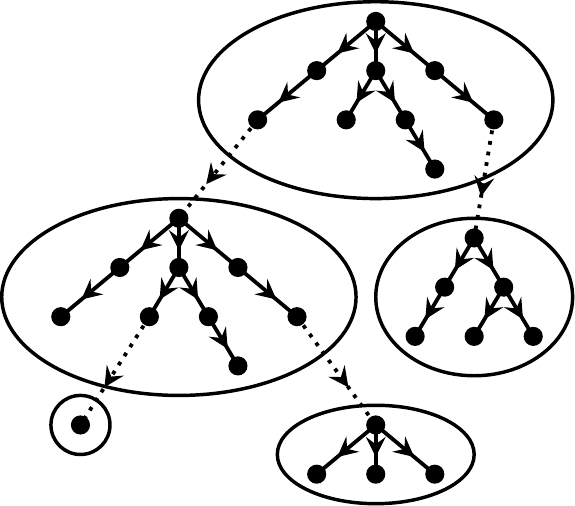}
                \label{fig:tree-split}
            \end{figure}

            Given a partition of $T$ into subtrees $T_1, \ldots, T_\ell$ as in Lemma \ref{lem:tree-split}, if we contract each subtree $T_i$ to a single vertex, the resulting graph $T'$ is again an out-directed tree with no multiple edges. We call this graph a \emph{(c,$\alpha$)-tree-split} of $T$; note that this split need not be unique. A subtree $T_i$ in such a split that does not have extending-leaves, i.e.\ the vertex corresponding to $T_i$ in $T'$ is a leaf, is called a \emph{leaf-tree}.

            \begin{proof}[ of Lemma \ref{lem:tree-split}]
                For each $i \le 2 m^{1-c\alpha}$ we construct a subtree $T_i$ in two stages. At step $i$ we assume we have already found $T_1,\ldots, T_{i-1}$ for which the conditions (\ref{itm:tree-split-one-in-edge})-(\ref{itm:tree-split-size}) hold, and $V(T_1) \cup \ldots \cup V(T_{i-1})$ induces a subtree $T'$ of $T$ with the same root. In the first stage we choose the root of $T_i$ and ensure that $T_i$ is big enough (or a leaf-tree), so that we are later able to deduce (\ref{itm:tree-split-number}), and in the second stage we extend $T_i$ further to ensure that it satisfies (\ref{itm:tree-split-one-in-edge})-(\ref{itm:tree-split-size}). We stop the process when all the vertices of $T$ are covered by the subtrees $T_1,\ldots, T_i$, and for such $i$ we denote $\ell:=  i$.

                \paragraph{Stage 1.} 
                    First we choose the root $v$ of $T_i$. For $i=1$ we take the root of $T$ and for $i> 1$ we pick the only out-neighbour of an extending-leaf (there is only one out-neighbour by (\ref{itm:tree-split-odd-level})) of $T'$ (this is the subtree of $T$ induced by $V(T_1) \cup \ldots \cup V(T_{i-1})$). 

                    Assume first that $v$ has at most $m^{c\alpha}$ descendants in $T$. Then we let $T_i$ be the subtree consisting of all descendants of $v$ in $T$. In this case $T_i$ is a leaf-tree of order at most $m^{c\alpha}+1$ and there is no second stage.

                    Otherwise, we start with a subtree $T_i'$ consisting only of the vertex $v$. As long as $|T_i'|<m^{c\alpha}$ we pick an extending-leaf of $T_i'$ and add all its children to $T_i'$. Note that such an extending-leaf always exists because there are more than $m^{c\alpha}$ descendants of $v$ in $T$ and each non-leaf vertex of $T_i'$ has all its children from $T$ in $T_i'$. Since the maximum out-degree of $T$ is bounded from above by the number of out-leaves of $T$ (every out-neighbour eventually leads to a different out-leaf by following out-edges), in each step of the construction of $T_i'$ we add at most $m^{\alpha}$ vertices. This implies that when we stop (i.e.\ right after $|T_i'| \ge m^{c\alpha}$ holds), we have the following.
                        $$ m^{c\alpha} \le |T_i'| \le  m^{c\alpha}+ m^\alpha \le 2m^{c\alpha}.$$

                \paragraph{Stage 2.} 
                    We start with $T_i=T_i',$ produced by stage 1. Call an extending-leaf contained in $T_i$ \emph{bad} if it lies in an odd level or has out-degree not equal to one in $T$. As long as there is a bad extending-leaf in $T_i,$ we add all its children to $T_i$. Eventually there are no bad extending-leaves left, since by going deep enough we reach a leaf of $T$ which by definition is not an extending-leaf.

                Note that during the procedure of both stages (\ref{itm:tree-split-one-in-edge}) is always satisfied by the choice of the root. Moreover, at the end of the procedure, (\ref{itm:tree-split-extend-leaf}) and (\ref{itm:tree-split-odd-level}) hold as well, since there are no bad extending-leaves in $T_i$.

                Now let us prove that condition (\ref{itm:tree-split-size}) holds, i.e.\ that $|T_i|\le 6 m^{c\alpha}$. From the first stage we know that $| T_i'|\le 2m^{c\alpha}$. Furthermore, every vertex in $T_i\setminus T_i'$ is either a leaf in $T$ or it was a bad vertex for some $T_i'$. In the latter case such a vertex is either branching or has out-degree $1$ in $T$ and is on an odd level, in the second case its child is either a branching vertex or a leaf of $T$ itself. Recall that there are at most $m^\alpha$ leaves in $T$ so by Lemma \ref{lem:branching}, there are at most $m^{\alpha}$ branching vertices. Finally, as each vertex has a unique parent the number of vertices of $T_i \setminus T_i'$ of the last type (i.e.\ vertices on odd levels of $T$ whose out-degree is $1$) is bounded by the number of leaves of $T$ plus the number of branching vertices of $T$. This implies that
                    $$ |T_i| = |T_i'| + |T_i\setminus T_i'| \le 2 m^{c\alpha} + m^{\alpha} + m^{\alpha} + 2m^\alpha \le 6 m^{c\alpha}.$$
          
                To see that the last condition (\ref{itm:tree-split-number}) holds, we note that each leaf-tree contains at least one out-leaf of $T$. Thus, the number of leaf-trees is bounded by $m^\alpha$. In addition to that we can bound the number of non-leaf-trees by $m/m^{c\alpha}$, since each one has order at least $m^{c\alpha}$. This implies that 
                        $$ \ell \le m^\alpha + m^{1-c\alpha} \le 2m^{1-c\alpha},$$
                where the last inequality follows from $c\ge 2$ and $0<\alpha \le (2c)^{-1}$.
            \end{proof}

        \subsubsection*{The $\alpha$-path-split}

            In the following lemma we are interested in a similar split, but this time we want the subtrees in the split to be paths. The graph obtained by contracting the paths in the following lemma will be called an $\alpha$-\emph{path-split}. We call a vertex of a tree a \emph{junction} if it is a leaf or a branching vertex.
            
            \begin{figure}[ht]
                \caption{The path-split of a tree.}
                \includegraphics{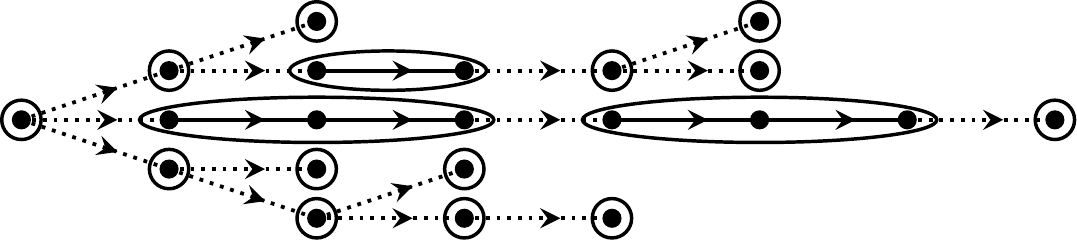}
                \label{fig:path-split}
            \end{figure}

            \begin{lem} \label{lem:path-split}
                Let $0<\alpha \le \frac{1}{4}$. Let $T$ be an out-directed tree on $m$ vertices with at most $m^\alpha$ leaves. Then there is a partition of the vertices into subpaths $P_1,\dots,P_\ell$ such that the following properties hold.
                \begin{enumerate}[(i)]
                    \item \label{itm:path-split-junction}
                        If $P_i$ contains a junction then $|P_i|=1$,
                    \item \label{itm:path-split-in-out-edge}
                        for every $i\in[\ell]$ there is at most one in-edge towards $P_i$, which is directed towards the start-vertex of $P_i$. Furthermore, unless $P_i$ is a junction, there is at most one out-edge away from $P_i$ which is directed from the end-vertex of $P_i$,
                    \item \label{itm:path-split-size}
                        $|P_i| \le m^{3\alpha}$ and
                    \item \label{itm:path-split-num}
                        $ \ell \le 5 m^{1-3\alpha}$.
                \end{enumerate}
            \end{lem}

            \begin{proof}
                We define the paths in the split as follows. First we let each junction be a separate trivial path of size $1.$ We then remove all junctions from the tree, thus we are left with a collection of disjoint subpaths, which we call \textit{long subpaths}. Finally, we split each such path into as few smaller subpaths, called \textit{short subpaths}, such that each has order at most $m^{3\alpha}$. We let these shorter paths be the remaining subpaths of our split. We now show that this split satisfies the desired conditions.
    
                First note that the number of junctions is at most $2m^\alpha$, since by assumption there are at most $m^\alpha$ leaves and therefore at most $m^\alpha$ branching vertices, by Lemma \ref{lem:branching}. If we consider a graph whose vertex set is the set of junctions and put an edge between a pair of junctions whenever they are joined by a long subpath, we obtain a forest. Therefore, the number of long subpaths, denoted by $d$, satisfies the following.
                    $$ d \le \text{ \# of junctions } - 1 \le 2 m^\alpha.$$
                For $i \in [d]$, denote by $m_i$ the order of the $i$-th long subpath. Then we split the $i$-th long subpath into $r_i := \ceil{\frac{m_i}{m^{3\alpha}}} \le \frac{m_i}{m^{3\alpha}} + 1$ shorter subpaths of order at most $m^{3\alpha}$ each. Hence, the total number of paths used is bounded from above by
                $$
                    \sum_{i \in [d]} r_i + 2m^{\alpha} \le \sum_{i \in d}\frac{m_i}{m^{3\alpha}} + d +2m^{\alpha} \le m^{1 - 3\alpha} + 4 m^{\alpha} \le 5m^{1 - 3\alpha},
                $$
                as required.
            \end{proof}
        \subsubsection*{The $k$-core-split}
    
            We now define the \emph{$k$-core-split} $F_1,\dots, F_\ell$ of a tree $T$; this will be used in Subsection \ref{ssec:general-tree-vs-path} to remove the requirement on the number of leaves of $T$ in order to prove Theorem \ref{thm:tree-vs-path}. Let $T'$ be the $k$-core of $T$. We set $F_1 = T'$. For $i>1,$ let $S_1, S_2, \dots, S_{\ell_i}$ be the trees of the forest $T\setminus \bigcup_{j\in [i-1]} F_j$. For every $j\in[\ell_i]$ we define $T_j$ to be the $k$-core of $S_j$ and set the forest $F_i :=  \bigcup_{j\in[\ell_i]} T_j$.
            
            \begin{figure}[ht]
                \caption{The $3$-core-split of a tree on $19$ vertices.}
                \includegraphics{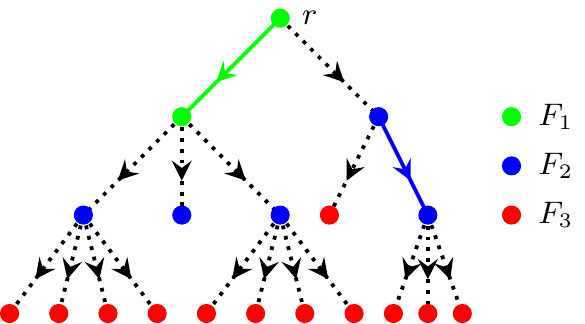}
                \label{fig:3-core-split}
            \end{figure}
            
            \begin{prop} \label{prop:core}
                Let $F_1,\dots, F_\ell$ be the $k$-core-split of a tree $T$. If $T'$ is a tree in the forest $F_i$ then
                \begin{enumerate}[(i)]
                    \item \label{itm:split-1}
                        $T'$ has order less than $\frac{|T|}{k^{i-1}}$ and               
                    \item \label{itm:split-2}
                        $\lf(T') \le k$.
                \end{enumerate}
            \end{prop}
            
            \begin{proof}
                Let us prove (\ref{itm:split-1}) by induction over $i$. The statement is clearly true for $i=1$. Let us assume that it holds for some $i < l$, then every tree $S$ from $F_i$ has size less than $\frac{|T|}{k^{i-1}}$ and is a $k$-core of some subtree of $T$. Therefore, every tree in $F_{i+1}$ has order at most  $\frac{|T|}{k^{i-1}} \cdot k^{-1} = \frac{|T|}{k^{i}}$, using Observation \ref{obs:core}, as desired. Property (\ref{itm:split-2}) follows, because every tree in $F_i$ is a $k$-core of some subtree of $T$.
            \end{proof}
    
    \subsection{Tree with few leaves}\label{ssec:few-leaves}
    
    We will make use of the following theorem (Theorem 3.17 in \cite{complete-directed-ramsey}). In fact, we will only use a special case of this theorem, when one of the trees is a path.

    \begin{thm} \label{thm:complete-ramsey}
        There exists a constant $c$ such that for any oriented trees $T_1$ and $T_2,$ in any 2-colouring of the complete directed graph on $c(|T_1|+|T_2|)$ vertices there exists a red $T_1$ or a blue $T_2$.
    \end{thm}

    Furthermore, our proof makes use of the special case of Theorem \ref{thm:tree-vs-path} for path vs.\ path proved in \cite{path-vs-path} as Theorem 12. 

    \begin{thm} \label{thm:path-vs-path}
        Given $0< \eps < \frac{1}{2}$ and $\sigma > 0$ there is a constant $c > 0$ such that the following holds. Let $G$ be an $(\eps, \sigma \log N)$-pseudorandom tournament on $N$ vertices. Then $G \to \left(\dirpath{n}, \dirpath{m}\right)$ provided $n,m \le N/ c$ and $nm \le \frac{N^2}{c ^2 \log N}$.
    \end{thm}
    This theorem is tight up to the constant factor. This result is the main bottleneck for our proof, as it will rely on this result as a black box.
    
    The following theorem is an intermediate result, leading to the proof of the Theorem \ref{thm:tree-vs-path}. Its proof introduces some interesting new ideas about how to deal with trees in place of paths.

    \begin{thm} \label{thm:bounded-tree-vs-path}
        Given $0<\eps< \frac{1}{2}$ and $\sigma > 0$, there exists a constant $c > 0$ such that the following holds. Let $G$ be a tournament on $N$ vertices, which is $(\eps, \sigma \log N)$-pseudorandom. Then $ G \rightarrow (\dirpath{n}, T)$, where $T$ is any directed tree on $m$ vertices, with at most $m^{1/6}$ leaves, as long as $n$, $m \le N/c$ and $nm \le \frac{N^2}{c^2\log N}$.
    \end{thm}
\vspace{-0.2cm}
    \begin{proof}    
        Let $\alpha = 1/6$. Without loss of generality we may assume that $T$ is out-directed, because once we prove this since directed paths are both in-directed and out-directed, we can apply it to the tournament with opposite orientation of every edge to conclude the other case.
        
        Consider a fixed $2$-colouring of $G$. Write $a= 128\eps^{-2},b=8a$ and make the following definition.
        $$
        \text{A cycle } C \text{ is }
            \left\{
            \begin{array}{ll}
                \text{\textit{short}} & \text{if }\, |C| < a m^{2\alpha},\\
                \text{\textit{medium}} & \text{if }\, a m^{2\alpha} \le |C| \le b m^{2\alpha} ,\\
                \text{\textit{long}} & \text{if }\,  b m^{2\alpha} < |C|.
            \end{array}
            \right.
        $$
        We will prove the theorem under assumptions that $N \ge c^2$ and $n, m \ge N /(c\log N),$ for some value of $c$. We start by arguing how to conclude the theorem in general assuming we know it under these conditions. To see this with only the first assumption, let $n' =\max (n, N / (c\log N))$ and $m' =\max (m, N / (c\log N))$. It is easy to see that $n',m' \le N/c$ and $n'm' \le \frac{N^2}{c^2\log N}$ still hold. The required result for $m$ and $n$ now follows from the result for $n'$ and $m'$, which satisfy both assumptions. So there is a $c$ such that the general result holds, given $N \ge c^2.$ Now if $N < c^2,$ then $c^2>N \ge cn,cm$ implying $n,m<c,$ so the result follows, with a larger constant, by appealing to the fact that oriented Ramsey numbers of acyclic graphs are finite. 
        
        So, from now on we assume $N \ge c^2$ and $n, m \ge N /(c\log N).$ In particular, this implies $m^\alpha \ge \sigma \log N.$ 
        
\vspace{-0.3cm}
        \subsubsection*{Case 1. many disjoint medium or long blue cycles}
            We shall consider two subcases. 
            \begin{itemize}
                \item[(1a)]  There is a collection of vertex-disjoint medium blue cycles covering at least $N/4$ vertices,
                \item[(1b)] There is a collection of vertex-disjoint long blue cycles covering at least $N/4$ vertices, such that they span no medium blue cycle and all the large blue cycles are as short as possible.
            \end{itemize}
            
            We resolve both subcases by finding the following structure. Given $k$ and $t$, a \emph{$(k,t)$-red-blue pairs} are a collection of pairwise disjoint subsets of vertices $A_1,B_1,\dots,A_t,B_t$, each of size $k$, such that for every $i\in [t]$ the following holds.
            \begin{enumerate}[(i)]
                \item The bipartite graph $G[A_i,B_i]$ contains only red edges,
                \item For every $i\in [t]$ there exists a blue path $P_i$ that contains all vertices of $A_i$, such that $P_i$ and $P_j$ are vertex-disjoint for all $j\neq i$.
            \end{enumerate}
            
            \begin{figure}[ht]
                \caption{Red-blue pairs}
                \includegraphics{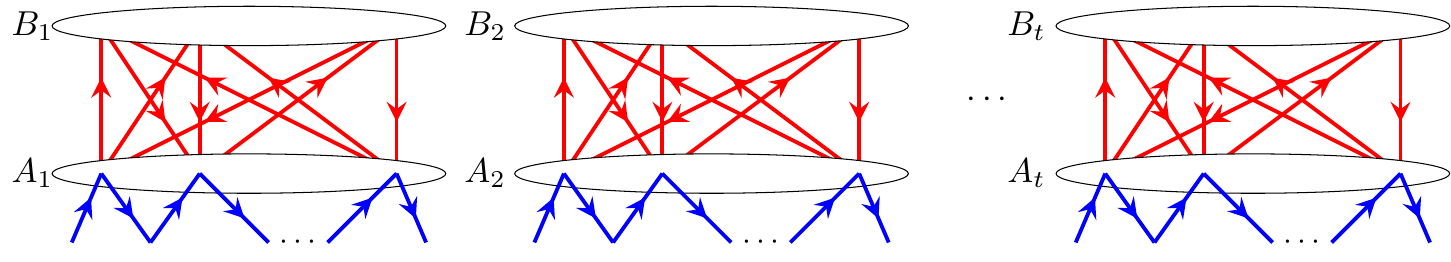}
                \label{fig:red-blue-pairs}
            \end{figure}
            
            We start by showing how to conclude the argument once we find this structure and then we show how to find it in each of the two cases.

            \begin{prop}\label{prop:red-blue pairs}
                Let $\frac{3}{4} am^{2\alpha} \le k \le am^{2\alpha}$ and $t\ge \frac{N}{256k}$. If $G$ contains $(k,t)$-red-blue pairs then $G \rightarrow (\dirpath{n}, T)$. 
            \end{prop}

            \begin{proof}
  
                Let us define an auxiliary complete directed graph $K$ on vertex set $[t]$. We colour the edge $ij$ blue if at least $(1-\frac{\eps}{4})k$ vertices in $A_i$ have at least $\frac{\eps}{2} k$ blue out-neighbours in $A_j$ and red otherwise.

                Theorem \ref{thm:complete-ramsey} implies that in the auxiliary graph $K$ there is a directed blue path of order $\frac{t}{2c_1}$ or any oriented red tree of order $\frac{t}{2c_1}$, for some positive constant $c_1$. We claim that the constant $c$ can be chosen such that the following two inequalities hold.
                \begin{enumerate}[(1)]
                    \item \label{itm:3.1}
                    $t\ge \frac{8c_1}{\eps k}n$,
                    \item \label{itm:3.2}
                    $t\ge 4 c_1 m^{1-2\alpha}$.
                \end{enumerate}
 
                Inequality (\ref{itm:3.1}) follows, as, by the assumptions of Proposition \ref{prop:red-blue pairs} and Theorem \ref{thm:bounded-tree-vs-path},  $t\ge \frac{N}{256k}$ and $ N \ge cn \ge 256\cdot 8\eps^{-1} c_1n$,  where the last inequality holds under the assumption that $c\ge 2^{11}c_1\eps^{-1}$. Inequality (\ref{itm:3.2}) follows from the assumption that $N\ge cm$ as $ t\ge \frac{N}{256k}\ge \frac{N}{256am^{2\alpha}} \ge \frac{c}{256a}m^{1-2\alpha} \ge 4 c_1 m^{1-2\alpha},$ given $c\ge 2^{10} a$.
    
                Suppose that there exists a blue path $P=i_1 i_2 \dots i_\ell$ of order $\ell := \frac{t}{2c_1}$ in $K$. We will explain how to lift this path to a blue path in $G$ by using subpaths of order at least $s:=\frac{\eps}{4}k$ from each of the blue paths $P_{i_j}$ associated to $A_{i_j}$ for $j \in [\ell]$.

                For every $j \in [\ell]$ let us denote by $E_j \subseteq A_{i_j}$ the set of vertices with at least $2s=\frac{\eps}{2}k$ blue out-neighbours in $A_{i_{j+1}}$. Note that $|E_j| \ge (1-\frac{\eps}{4})k$, since the edge $i_j i_{j+1}$ is blue in $K$.

                We start with an initial subpath $P'$ of the path $P_{i_1}$ associated to $A_{i_1}$ which ends in the last vertex from $P_{i_1}$ contained in $E_1$. Note that $|P'|\ge |E_1|\ge s$, since $P_{i_1}$ covers $A_{i_1}$, so also $E_1$.

                Suppose that we already have a path $P'$ of order at least $s(j-1)$ in  $\bigcup_{r\in [j-1]} A_{i_r}$ whose last vertex $v$ is contained in $E_{j-1}$.

                Let $S_{j}$ be the blue out-neighbourhood of $v$ in $A_{i_j}$ and denote by $u_1,u_2, \dots, u_{k}$ the vertices of the blue path $P_{i_j}$ in $A_{i_j}$ ordered according to their order in $P_{i_j}$. We extend $P'$ by the path $Q$ consisting of $u_p, u_{p+1},\dots, u_{q}$, where $p$ is the smallest index among the vertices in $S_j$ and $q$ the largest index among the vertices in $E_j$. Note that $Q$ has order at least $s$, since
                \begin{align*}
                    |E_j \cap S_j| &= |E_j| + |S_j| - |E_j \cup S_j| \\
                    &\ge |E_j|+|S_j| - |A_{i_j}| \\
                    &\ge \left(1-\frac{\eps}{4}\right)k + \frac{\eps}{2}k - k = \frac{\eps}{4}k=s.
                \end{align*}

                The blue path in $K$ has order $\frac{t}{2c_1}$, so this process produces a blue path in $G$ of order at least $\frac{ts}{2c_1} \ge \frac{\frac{8c_1n}{\eps k}\cdot \frac{\eps k}{4}}{2c_1}= n$, where we used inequality (\ref{itm:3.1}); this completes the proof of Proposition \ref{prop:red-blue pairs} in the case where $K$ contains a long blue path.

                Otherwise, i.e.\ if $K$ does not have a blue path of order $\frac{t}{2c_1}$, then $K$ contains a red copy of a $(2,\alpha)$-tree-split of $T$ (obtained from Lemma \ref{lem:tree-split}), since $\frac{t}{2c_1}\ge 2m^{1-2\alpha},$ by inequality (\ref{itm:3.2}). We now explain how to lift the tree-split from $K$ to a red copy of $T$ in $G$. We relabel the vertices of $K$ in such a way that the vertices of the tree-split, we found in $K,$ are $[\ell]$, for $\ell$ being the order of the tree-split, and vertex $i$ representing a subtree $T_i$ of $T$. Our aim is to embed the subtrees $T_i$ in $A_i\cup B_i.$ To that end we pick `candidate sets' $D_i \subseteq A_i$ which satisfy some useful properties.

                \begin{claim} \label{claim:lift-tree}
                    Suppose that for each vertex $i$ of the tree-split we have a non-empty set of candidates $D_i \subseteq A_i$, such that for any $v \in D_i$ there is a tree $T(v)$ with the following properties.
                    \begin{enumerate}[(i)]
                        \item \label{itm:T-i}
                    	    $T(v)$ is a red copy of $T_i$ embedded in $A_i \cup B_i$ and rooted at $v$,
                        \item \label{itm:edge-D-j}
                    	    each vertex $u$ in $T(v)$ that corresponds to an extending-leaf $w$ in $T_i$ (as a subtree of $T$) is in $A_i$ and has a red out-edge towards $D_j$ if $j$ is such that there is an edge from $w$ to $T_j$ in $T$.
                    \end{enumerate}
                    Then we can find a red copy of tree $T$ inside the tournament $G$.
                \end{claim}

                \begin{proof}
                    Let us denote by $T'$ the subtree of the tree-split containing the root $v$ of $T$. From the set of candidates for the root of $T'$ we can pick any vertex we want and set $T' := T(v)$. By property (\ref{itm:edge-D-j}) we can choose the roots of the adjacent subtrees in the corresponding candidate sets and by (\ref{itm:T-i}) we can embed the subtrees themselves as well. Note that as all $A_i$ and $B_i$ are disjoint, we do not use any of the vertices twice. Repeating this argument eventually produces a red copy of $T$ in the tournament $G$.
                \end{proof}

                We now show how to construct appropriate candidate sets consisting of at least $k/2$ vertices. For this we begin with the leaves of the tree-split and then make our way up, in the sense that we deal with the candidate set of a particular vertex from the tree-split only if we already defined the candidate sets for all its out-neighbours. 

                Let us define the candidate set $D_i$, where the out-neighbours of vertex $i$ in the tree-split are $j_s$ for $s \in[h]$ (if $T_i$ is a leaf-tree then $h=0$). Note that by the assumption on the number of leaves of the tree $T$ and the definition of a tree-split, $h \le m^{\alpha}$. Furthermore, by construction, the candidate sets $D_{j_s}$ have already been defined (this condition also holds, in particular, for leaf-trees).

                For each $s \in [h]$, let $X_{s}\subseteq A_i$ be the set of vertices with at least one red out-neighbour in $D_{j_s}$. These sets will host the extending-leaves and guarantee property (\ref{itm:edge-D-j}) of Claim \ref{claim:lift-tree}. Let $Y$ be the set of vertices in $B_i$ that send at least $|T_i| + \sigma \log N$ red edges into each set $X_s$; if there are no such sets $X_s$ (i.e.\ $T_i$ is a leaf-tree) we let $Y$ be the set of vertices in $B_i$ with at least $|T_i| + \sigma \log N$ red out-neighbours in $A_i$. Finally, let $D_i$ be the set of vertices in $A_i$ that send at least $|T_i|$ red edges into $Y$.
                \begin{claim} \label{claim:size-D-i}
            		$|D_i| \ge |A_i| - \sigma \log N \ge k/2$.
                \end{claim}
            
            	\begin{proof}
            		Firstly, we show that $|X_s| \ge \frac{\eps}{8}k$ for every $s \in [h]$.
            		Indeed, as $|D_{j_s}|\ge k/2 \ge \sigma \log N$, by pseudorandomness, all but at most $\sigma \log N$ vertices of $A_i$ send at least $\eps |D_{j_s}| \ge \frac{\eps}{2}k$ edges into $D_{j_s}$. Since $i j_s$ is a red edge in the auxiliary graph $K$, at most $(1 - \frac{\eps}{4})k$ of these vertices send at least $\frac{\eps}{2}k$ blue edges into $D_{j_s}$. It follows that there are at least $|A_i| - (1 - \frac{\eps}{4})k - \sigma \log N \ge \frac{\eps}{8}k$ vertices in $A_i$ with at least one red out-neighbour in $D_{j_s}$, i.e.\ $|X_s| \ge \frac{\eps}{8}k$, as claimed. 
            		
            		We now claim that $|Y| \ge k/2$. Indeed, since all edges between $A_i$ and $B_i$ are red, by pseudorandomness, all but at most $\sigma \log N$ vertices of $B_i$ send at least $\eps |X_s| \ge \frac{\eps^2}{8}k \ge |T_i| + \sigma \log N$ red edges into $X_s$, for every $s \in [h]$. Hence, $|Y| \ge |B_i| - h \cdot \sigma \log N \ge k - m^{\alpha} \cdot \sigma \log N \ge k/2$. A similar argument shows that $|Y| \ge k/2$ if $T_i$ is a leaf-tree.
            		
            		Finally, by pseudorandomness and since all edges from $A_i$ to $B_i$ are red, we find that $|D_i| \ge |A_i| - \sigma \log N \ge k/2$, as required.
            	\end{proof}

                Now let us explain why each vertex in $D_i$ is a candidate for the root of $T_i$. Note that the bipartite graph $G[D_i, Y]$ has minimum red out-degree at least $|T_i|$, allowing us to greedily embed a copy $T'$ of the subgraph of $T_i$ obtained by removing its extending leaves. By the property that all extending leaves lie at even distance from the root of their corresponding subtree, the parent of such a leaf $u$ is embedded in $Y$ and has at least $|T_i|$ red out-neighbours in the corresponding set $X_{s}$. Since we have not embedded all the vertices of $T_i$ yet, we can embed $u$ in $X_s$. 

                These candidate sets $D_i$ satisfy conditions of Claim \ref{claim:lift-tree}, so we can find a red copy of the desired tree in $G$. This completes the proof of Proposition \ref{prop:red-blue pairs}.
            \end{proof}

            In order to complete the proof of Theorem \ref{thm:bounded-tree-vs-path} in Case 1, it now remains to show how to find $(k,t)$-red-blue pairs in the Cases (1a) and (1b).

            \subsubsection*{Case 1a. many disjoint medium blue cycles}
                We assume that there is a collection of vertex-disjoint medium blue cycles $C_1, C_2, \dots, C_{t'}$ which cover at least $N/4$ vertices. Since the medium cycles have length at most $bm^{2\alpha}$, we get that $t'\ge \frac{N}{4 bm^{2\alpha}}$.

                Let $H$ be an auxiliary 2-coloured complete directed graph on vertex set $[t']$. We colour the edge $ij$ blue if at least $\frac{a}{4} m^{2\alpha}$ vertices in $C_i$ have a blue out-neighbour in $C_j$ and red otherwise. Now we consider a maximal red-red matching $M$, namely a matching that consists of edges that are red in both directions.

                First we suppose that the matching $M$ covers at most $t'/2$ vertices. Since this matching is maximal, for every two vertices $i$ and $j$ not covered by $M$ at least one of the directed edges $ij$ and $ji$ is blue.  In particular, there is a blue subtournament on at least $t'/2$ vertices. Since every tournament contains a directed Hamiltonian path, we thus find a blue directed path of order $t'/2$ in the auxiliary graph $H$. The following claim explains how to lift this path to the tournament $G$.

                \begin{claim} \label{claim:lift-path}
                    Let $G$ be an oriented graph with pairwise vertex-disjoint cycles $C_1,C_2,\dots,C_k$ such that for each $i<k$ there are at least $r$ vertices in $C_i$ that have an out-neighbour in $C_{i+1}$. Then $G$ contains a directed path of order $k \cdot r$.
                \end{claim}
                \begin{proof}
                    We start with a path $P'$ which begins at an arbitary vertex of $C_1$ and follows the cycle $C_1$ up to the last vertex that sends an edge towards $C_2$. Note that $|P'|\ge r$.

                    Assume that $P'$ is a path of order at least $r(i-1)$ with vertices in $\bigcup_{j\in [i-1]} C_j$, such that its last vertex has an out-neighbour $w$ in $C_i$. We extend $P'$ by the path $Q$ starting at $w$ and following $C_i$ up to the last vertex which sends an edge to $C_{i+1}$, or until the last vertex in $C_i$ if $i=k$. Note that $|Q|\ge r$ and therefore the new path $P'$ has order at least $r\cdot i$.
                \end{proof}

                By applying Claim \ref{claim:lift-path} to the induced blue subgraph of $G$ we find a blue path of order at least
            	    $$\frac{a}{4}m^{2\alpha}\cdot t'/2 \ge \frac{a}{32b}N \ge \frac{N}{256}\ge n,$$
                (where the last inequality holds by assuming that $c\ge 256$) as desired.

                Therefore, we can assume that the matching $M$ covers at least $t'/2$ vertices of $H$. This corresponds to $t:= t'/4$ disjoint pairs of cycles $(C_{i_s},C_{j_s})$, for $s \in [t]$, where at least $k:= \frac{3}{4}am^{2\alpha}$ vertices in $C_{i_s}$ do not have a blue out-neighbour in $C_{j_s}$ and vice versa. Hence we can find subsets $A_s$ and $B_s$ of $C_{i_s}$ and $C_{j_s}$, respectively, of size $k$ each, with only red edges between them. Each $A_s$ lies in a different cycle $C_{i_s}$ from the original collection of disjoint medium blue cycles. Thus, we can define $P_s$ to be the cycle $C_{i_s}$ minus one edge. This way $P_s$ contains all vertices of $A_s$ and the paths $P_s$ are pairwise disjoint.

                Note that $t\ge \frac{3N}{512k}$, since $t'\ge \frac{N}{4bm^{2\alpha}} \ge \frac{N}{32am^{2\alpha}} = \frac{3N}{128k}. $ So Proposition \ref{prop:red-blue pairs} applies and concludes the proof of Theorem \ref{thm:bounded-tree-vs-path} in this case.

        \subsubsection*{Case 1b. a large set with many disjoint long blue cycles but no blue medium cycle}
            Suppose Case 1a does not hold, thus there exists a set $U$ of at least $3N/4$ vertices which does not contain any medium blue cycle.

            Let us consider the following process which starts with $U':= U$. As long as there exists a long blue cycle in $U'$ we pick a shortest one, say $C$ and define $U'=U'\setminus C$.
            This process eventually terminates and produces a sequence of disjoint long blue cycles $C_1,C_2,\dots,C_{t'}$. 
            In this case we are going to assume that these cycles cover at least $N/4$ vertices. Note that $t' \le \frac{N}{4bm^{2\alpha}}$, since each long cycle contains at least $bm^{2\alpha}$ vertices.

            Note that for every $i\in [t']$ all chords in $C_i$ of length at least $am^{2\alpha}$ are red, since otherwise we would find a blue cycle inside $C_i$ which is either a medium cycle, or a shorter long cycle, contradicting our choice of $C_i$ as the shortest remaining long cycle.

            Write $C_i = (v_1 v_2 \dots v_r)$ and $k:= am^{2\alpha}$.
            Define sets
            $$
                A_i=\{v_1,v_2,\dots,v_{r/2-2k}\} \quad \text{and} \quad B_i=\{v_{r/2-k},v_{r/2-k+1},\dots,v_{r-k}\}.
            $$
            By the argument above we have that $G[A_i,B_i]$ spans red edges only. 
            Let $A_{i,1}, \ldots, A_{i,r(i)}$, where $r(i) = \floor{|A_i|/k}$, be pairwise disjoint sets of $k$ consecutive vertices (with respect to $C_i$) in $A_i$, and let $B_{i,1}, \ldots, B_{i, r(i)} \subseteq B$ be defined similarly.  The sets $A_{i,j}, B_{i,j}$ cover all but at most $4k$ vertices of each cycle $C_i$, hence they cover at least $N/4-t'\cdot 4k$ vertices in total. Since the number of sets $A_{i,j}$ and sets $B_{i,j}$ is the same, and each set has size $k$, it follows that the number $t$ of sets $A_{i,j}$ satisfies 
            	$$t \ge \frac{N}{8k} - 2t' \ge \frac{N}{8k} - \frac{N}{2bm^{2\alpha}} = \frac{N}{8k} - \frac{N}{16k} = \frac{N}{16k},$$
            	where the first equality follows from the choice $b=8a$ which implies that $bm^{2\alpha}=8k.$ Note that the collection of pairs $(A_{i,j}, B_{i,j})$ forms a $(k,t)$-red-blue pairs structure, as each set $A_{i,j}$ contains a spanning blue path (which is a part of the cycle $C_i$), which are mutually disjoint. Proposition \ref{prop:red-blue pairs} can now be used to complete the proof of Theorem \ref{thm:bounded-tree-vs-path} in this case as well.
 
        \subsubsection*{Case 2. a large set of vertices spanning no blue medium or long cycle}

            In the remaining case, the process of picking the sequence of disjoint long blue cycles in Case $1b$ terminated before it covered at least $N/4$ vertices. Hence, we are left with a set $U$ that covers at least $N/2$ vertices and spans neither medium nor long blue cycles.

            Let us start this case with an elementary observation.
            \begin{obs} \label{obs:cycle-by-degree}
                Every directed graph $G$ with minimum out-degree $d$ contains a cycle of length at least $d+1$.
            \end{obs}

            \begin{proof}
              Let $v_1  \dots v_\ell$ be a longest directed path in $G$. By the maximality of this path we get that $v_\ell$ has no out-neighbour outside of this path. Since the out-degree of $v_\ell$ is at least $d$ it has at least $d$ out-neighbours among $v_1, \dots, v_{\ell-1}$. Let $s$ be the smallest index among these out-neighbours of $v_\ell$. Then $(v_s v_{s+1} \dots v_\ell$) is a cycle of length at least $d+1$.
            \end{proof}

            This observation allows us to obtain an ordering of the vertices in $U$ with `few' blue edges going forward.
            \begin{claim} \label{claim:ordering}
                There exists an ordering $u_1,u_2, \dots, u_{|U|}$ of the vertices in $U$ such that for every $i$ there exists at most $am^{2\alpha}$ indices $j>i$ such that there is a blue edge from $u_i$ to $u_j$.
            \end{claim}

            \begin{proof}
                Suppose that there exists a subgraph of $G[U]$ which has minimum blue out-degree at least $a m^{2\alpha}$.
                Then by Observation \ref{obs:cycle-by-degree} we find a blue cycle of order at least $a m^{2\alpha}$, a contradiction. Therefore, there exists a vertex $u_1\in U$ with blue out-degree at most $a m^{2\alpha}$.
                Now suppose that $u_1,u_2,\dots,u_{i-1}$ are defined. In a similar way we define $u_i$ to be a vertex with blue out-degree at most $a m^{2\alpha}$ in $G[U']$, where $U'=U\setminus \{u_1,u_2,\dots,u_{i-1}\}$.
                We repeat this as long as $i\le |U|$. The resulting ordering $u_1,u_2,\dots,u_{|U|}$ satisfies the requirement of the claim.
            \end{proof}

            Let $k= \frac{N}{32m^{1-3\alpha}}$. We set $t:= |U|/k$ and denote $U_i=\{ u_{(i-1)k+1},\dots,u_{ik}\}$ for $i \in [t]$. We claim that we can choose the constant $c$ such that the following two inequalities hold.
            \begin{enumerate}[(1)]
                \item \label{itm:2.1} $t \ge 16m^{1-3\alpha}$,
                \item \label{itm:2.2} $k \ge 128 \eps^{-2}am^{3\alpha}$.
            \end{enumerate}
            Indeed, inequality (\ref{itm:2.1}) follows independently from $c$, as $t=\frac{|U|}{k}\ge \frac{N}{2k}=16m^{1-3\alpha}.$ We obtain inequality (\ref{itm:2.2}) from $N \ge cm,$ given $c\ge 2^{12}\eps^{-2}a$, as $k=\frac{N}{32m^{1-3\alpha}} \ge \frac{cm}{32m^{1-3\alpha}} = \frac{c}{2^{5}} m^{3\alpha} \ge 2^7\eps^{-2}am^{3\alpha}.$

            Let $c_2$ be the constant from Theorem \ref{thm:path-vs-path} with parameters $\eps$ and $\frac{\sigma}{3\alpha}$.
            By choosing $c\ge 128 c_2,$ we obtain from $N\ge cn, cm$  that $k/4=\frac{N}{128m^{1-3\alpha}}\ge c_2 \frac{n}{m^{1-3\alpha}}, c_2 m^{3\alpha}$.
            Similarly, we get $k/4 \ge c_2 \sqrt{ \frac{n}{m^{1-3\alpha}} m^{3\alpha} \log(k/4)},$ from $N \ge c \sqrt{nm\log N}$.
            
            Also note that $k/4 =\frac{N}{128m^{1-3\alpha}} \ge N ^{3 \alpha}$ so $ \sigma \log (k/4) \ge 3 \alpha \sigma \log N,$ implying that any subtournament of $G$ of order $k/4$ is $(\eps, \frac{\sigma}{3\alpha} \log(k/4))$-pseudorandom. Therefore, Theorem \ref{thm:path-vs-path} applies for paths of order $\frac{n}{m^{1-3\alpha}}$ and $m^{3\alpha}$, within any subset of vertices of size at least $k/4.$
            
            \begin{claim} \label{claim:red-blue-sets}
                One of the following holds, for each set $U_i$.
                \begin{enumerate}[(i)]
                    \item \label{itm:blue}
                        There are at least $k/8$ pairwise disjoint pairs of vertices in $U_i$ that are joined by a blue path, contained in $U_i$, of order $\frac{n}{m^{1-3\alpha}}$,
                    \item \label{itm:red}
                        For each $2 \le \ell \le m^{3\alpha}$, there are at least $k/4$ pairwise disjoint pairs of vertices in $U_i$ that are joined by a red path, contained in $U_i$, of order $\ell$.
                \end{enumerate}
            \end{claim}
            \begin{proof}
                Consider the following process. As long as there is a blue path of order $\frac{n}{m^{1-3\alpha}}\ge 2$ in $U_i$ (where the inequality follows since $n \ge \frac{N}{c\log N}$ and $m\le N/c$) we remove its first and last vertex. If this process runs for at least $k/8$ rounds then (\ref{itm:blue}) holds. 

                Otherwise, there is a subset $W \subseteq U_i$ of size at least $\frac{3}{4}k$ with no blue path of order $\frac{n}{m^{1-3\alpha}}$. Consider the following process. As long as there are $k/4$ vertices left in $W$ we can apply Theorem \ref{thm:path-vs-path} to find a red path of order $\ell$ (since $\ell \le m^{3\alpha}$) and remove its first and last vertex. Since we remove only two vertices in each round this process runs for at least $k/4$ rounds. Thus (\ref{itm:red}) holds.
            \end{proof}
            If (\ref{itm:blue}) holds, we say that $U_i$ is blue; otherwise, we say that $U_i$ is red. We now distinguish two cases depending on the majority colour of the sets $U_i$.

            \subsubsection*{Case 2a. most of the sets $U_i$ are red}
                In this case there are at least $t/2$ red sets $U_i$ so, while preserving the ordering, we rename $t/2$ red sets $U_i$ as $V_1,V_2,\dots, V_{t/2}$. Note that when $i<j$ we have by Claim \ref{claim:ordering} that every vertex in $V_i$ has at most $a m^{2\alpha}$ blue out-neighbours in $V_j$. Let us view $V_1,V_2,\dots V_{t/2}$ as vertices of a transitive tournament with edges pointing always towards the bigger index. Let $T'$ be an $\alpha$-path-split of $T$ (see Lemma \ref{lem:path-split}). By Observation \ref{obs:transitive}, we can find a copy of $T'$ inside this transitive tournament, since inequality (\ref{itm:2.1}) implies that $t/2 \ge 5m^{1-3\alpha}$.

                We now show that if we define appropriate candidate sets for each start-vertex of a path in the path-split, than we can greedily find a red copy of $T$ in $G,$ in a similar manner as in Case 1. Let us denote by $P_i$ the path corresponding to the vertex $i$ of the embedded path-split.

                \begin{claim}\label{claim:lift-tree-v2}
                    Suppose that for each vertex $i$ of the path-split we have a non-empty set of candidates $D_i \subseteq V_i$, such that for any $v \in D_i$ there is a subpath $P(v)$ of $T$ which satisfies
                    \begin{enumerate}[(i)]
                        \item $P(v)$ is a red copy of $P_i$ embedded within $V_i$ with start-vertex $v$,
                        \item the end-vertex $u$ of $P(v)$ has a red out-edge towards $D_j$ for each $j$ which is a child of $i$ in the path-split.
                    \end{enumerate}
                    Then we can find a red copy of tree $T$ inside the tournament $G$.
                \end{claim}

                \begin{proof}
                    Use a greedy embedding, analogous to the one used in the proof of Claim \ref{claim:lift-tree}.
                \end{proof}
                We now define such candidate sets, each of size at least $k/8$. We start with the leaves of the path-split and then move upwards, in such a way that we are always defining the candidate set for the vertex whose all out-neighbours have already had their candidate sets defined.

                If $i$ is a leaf of the path-split, then $P_i$ is a leaf of $T$, and we can set $D_i:= V_i$. 
                
                In the case of $i$ being a non-leaf we apply Claim \ref{claim:red-blue-sets} with $\ell = |P_i|$ and define $E_i,S_i \subseteq V_i$ to be the sets of end and start-vertices of a red path of length $|P_i|,$ such that $|E_i|=|S_i| \ge k/4$ (note that if $P_i$ is a singleton, then we can take $E_i = S_i = V_i$). We distinguish two cases for each non-leaf $i$ in the path-split, depending on whether $i$ is a branching vertex of $T'$ or not.

                Suppose that $P_i$ corresponds to a non-branching vertex of $T'$. Then its end-vertex has out-degree exactly $1$ in $T$; denote this out-neighbour by $j$. Let $X$ be the subset of $E_i$, consisting of vertices that have at least one red out-neighbour in the candidate set $D_j$. We define the candidate set $D_i$ to be the set of vertices in $S_i$ that correspond to the vertices in $X$. In this case it remains to show that $|X|\ge k/8.$ Indeed, by pseudorandomness, all but at most $\sigma \log N $ vertices in $E_i$ send at least $\eps |D_j| \ge \eps k/8 > a m^{2\alpha}$ edges to $D_j$. Recall that by the choice of the ordering of the vertices, every vertex in $E_i$ has at most $a m^{2 \alpha}$ blue out-neighbours in $D_j$, hence all but at most $\sigma \log N$ vertices in $E_i$ have a red out-neighbour in $D_j$, i.e.\ $|X| \ge |E_i| - \sigma \log N \ge k/8$.

                Now suppose that $P_i$ is a branching vertex in the path-split, i.e.\ it corresponds to a branching vertex $v$ in $T$. The maximum out-degree of $T$ is bounded by the number of leaves, so $i$ has at most $m^{\alpha}$ out-neighbours in $T'$; denote them by $j_1, \ldots, j_h$ (so $h \le m^{\alpha}$). Let $D_i$ be the set of vertices in $V_i$ which have a red out-neighbour in each of the sets $D_{j_s}$ for $s \in [h]$. As before, all but at most $\sigma \log N$ vertices in $V_i$ have at least one red out-neighbour in $D_{j_s}$ for each $s$. Hence $|D_i| \ge |V_i| - h \cdot \sigma \log N \ge k/8$.

                We defined candidate sets required by Claim \ref{claim:lift-tree-v2}, so in the case when most of the sets $U_i$ are red, we find a red copy of $T$.

            \subsubsection*{Case 2b. most of the sets $U_i$ are blue}
                In this case we assume that at least $t/2$ of the $U_i$ are blue; let us now rename $t/2$ blue $U_i$'s as $V_1,V_2,\dots,V_{t/2}$, while preserving the ordering, and let $E_i, S_i \subseteq V_i$ be the sets of end and start-vertices of the (blue) paths given by Claim \ref{claim:red-blue-sets}; then $|E_i|=|S_i|\ge k/8$ for every $i$.

                Define an auxiliary complete directed graph $K$ on vertex set $[t/2]$, where vertex $i$ corresponds to $V_i$. We define the following 2-colouring of its edges. Every edge $ij$ with $i<j$ is coloured red. We colour an edge $ij$ with $i>j$ blue if for every choice of subsets $W_i \subseteq V_i$ and $W_j \subseteq V_j$ of size at least $k/16$, there is a blue edge from $W_i$ to $W_j$; otherwise, we colour the edge red.

                Let $M$ be a maximal red-red matching in $K$. We now distinguish two cases: $M$ covers at least $t/4$ of the vertices of $K$; or there is a blue directed path of order at least $t/4$ (we have seen in Case 1a that one of these possibilities occurs).

                \subsubsection*{There is a long blue path in $K$} \label{ssec:tree-vs-path}
                    In this case, we assume that there is a blue path $i_1 i_2 \ldots i_{t/4}$ in $K$. Let us denote by $X_j \subset E_{i_j}$ the set of vertices which are end-vertices of blue paths of order at least $j \cdot \ell$ in $\bigcup_{r\in [j]} V_{i_r}$, where $\ell :=  \frac{n}{m^{1-3\alpha}}.$

                    \begin{claim} \label{claim:lift-path-v2}
                        For every $j \in [t/4]$ we have $|X_j| \ge k/16$.
                    \end{claim}
                    \begin{proof}
                        We prove this by induction. In the case $j=1$ every vertex in $E_{i_1}$ is an end-vertex of a path of order $\ell$ in $V_{i_1}$. So, let us assume that the statement is true for some $j\ge 1.$ Let $Y_{j+1}\subseteq S_{i_{j+1}}$ be the set of vertices that have a blue in-neighbour in $X_j$.
                
                        We now show that $|X_{j+1}| \ge |Y_{j+1}|$. Let $v\in Y_{j+1}\subseteq S_{i_{j+1}}$ and $u\in E_{i_{j+1}}$ be its corresponding end-vertex of a path $Q$ of order $\ell$. Since $v \in Y_{j+1}$, there exists a vertex $w$ in $X_j$ such that the edge $wv$ is blue in $G$. By definition of $X_j,$ $w$ is the end-vertex of a path $P$ of order at least $j\cdot \ell$. Then $PwvQ$ is a path of order at least $(j+1)\cdot \ell$ in $\bigcup_{r \in [j+1]} V_{i_r}$, hence $u\in X_{j+1}$. This shows that $|X_{j+1}| \ge |Y_{j+1}|$.

                        As there are no blue edges between $X_j$ and $S_{i_{j+1}}\setminus Y_{j+1}$ and since $|X_j|\ge k/16$ (by induction) we have $|S_{i_{j+1}}\setminus Y_{j+1}|< k/16,$ by the definition of the auxiliary graph $K$. This implies that $|Y_{j+1}| > k/8-k/16 \ge k/16,$ so $ |X_{j+1}| \ge|Y_{j+1}| \ge k/16$, as required.
                    \end{proof}

                    By applying Claim \ref{claim:lift-path-v2} with $j = t/4$, we find a blue path of order $t/4 \cdot \frac{n}{m^{1-3\alpha}}$ in $G$. By inequality (\ref{itm:2.1}) this blue path has order at least $n$, as desired.

                \subsubsection*{There is a large red-red matching in $K.$}

                Now we consider the case where there is a red-red matching $M,$ which covers $t/4$ vertices of $K$. Let us denote the edges of $M$ by $(V_{i_1},V_{j_1}),(V_{i_2},V_{j_2}),\dots,(V_{i_{t/8}},V_{j_{t/8}})$, where $i_s < j_s$ for every $s \in [t/8]$ and $i_1 < \ldots < i_{t/8}$. By definition of the auxiliary graph $K$, there are subsets $A_s \subseteq V_{i_s}$ and $B_s \subseteq V_{j_s}$ of size $k/16$ each, such that all edges from $B_s$ to $A_s$ are red; fix such subsets. The vertices $i_1, \ldots, i_{t/8}$ form a red transitive tournament that respects this ordering, i.e.\ $i_s i_r$ is an edge if $s < r$. By Observation \ref{obs:transitive}, we may find within this tournament a copy of a $(3,\alpha)$-tree-split $T'$ of $T$, since $T'$  is an out-directed tree of size smaller than $ 2m^{1-3\alpha} \le t/8$ (by inequality (\ref{itm:2.2})).
                
                It remains to find appropriate candidate sets $D_i$ for each vertex $i$ from the tree-split so that we can find a red copy of $T$ in $G,$ by Claim \ref{claim:lift-tree}. We will construct $D_i$ such that they have size at least $k/32$. As before, we start with leaf-trees and work our way up the tree, in such a way that when we are about to define a candidate set $D_i$, the candidate sets of subtrees corresponding to out-neighbours of the vertex $i$ in the tree-split are already defined. 
            
			    Let $j_1, \ldots, j_h$ be the out-neighbours in the tree-split of a vertex $i$. We assume that $D_{j_s} \subseteq A_{j_s}$ has been defined and has size at least $k/32$. Note that $h \le m^{\alpha}$ due to the bound on the number of leaves of $T$, and possibly $h = 0$ if $i$ corresponds to a leaf-tree. Let $X_s$ be the set of vertices in $A_i$ which have at least one red out-neighbour in $D_{i_s}$, for $s \in [h]$. Let $Y$ be the set of vertices in $B_i$ which have at least $|T_i| + \sigma \log N$ red out-neighbours in $X_s$ for every $s \in [h]$; if $h = 0$ we define $Y$ to be the set of vertices in $B_i$ that have at least $|T_i| + \sigma \log N$ red out-neighbours in $A_i$. Finally, let $D_i$ be the set of vertices in $A_i$ that have at least $|T_i|$ red out-neighbours in $Y$.
			
			    \begin{claim}
				    $|D_i| \ge |A_i| - \sigma \log N$.
			    \end{claim}
			
			    \begin{proof}
				    Firstly, note that every vertex in $A_i$ has at most $a m^{2\alpha}$ blue out-neighbours in $A_{j_s}$ (as $j_s$ corresponds to a set that appears later in the ordering $V_1, \ldots, V_{t/8}$ than the set that contains $A_i$). It follows from pseudorandomness that all but at most $\sigma \log N \le k/32$ vertices in $A_i$ have at least $\eps |A_{j_s}| >  a m^{2\alpha}$ out-neighbours in $A_{j_s}$, at least one of which is red. In particular, $|X_s| \ge k/32$.
				 
				    Secondly, again by pseudorandomness and by the fact that all edges from $B_i$ to $A_i$ are red, all but at most $\sigma \log N$ vertices in $B_i$ have at least $\eps |X_s| \ge |T_i| + \sigma \log N$ red neighbours in $X_s$. It follows that $|Y| \ge |B_i| - h \cdot \sigma \log N \ge k/32$. If $h = 0$ then, similarly, $|Y| \ge k/32$.
				
				    Finally, recall that the vertices in $A_i$ have at most $a m^{2\alpha}$ blue out-neighbours in $B_i$. Hence, by pseudorandomness, all but at most $\sigma \log N$ vertices in $B_i$ have at least $\eps |Y|$ out-neighbours in $Y$, at least $\eps |Y| - a m^{2 \alpha} \ge |T_i|$ of which are red. It follows that $|D_i| \ge |A_i| - \sigma \log N$, as required. 	
			    \end{proof}

                By Claim \ref{claim:lift-tree} we may find a red copy of $T$ in $G$. This completes the proof of Theorem \ref{thm:bounded-tree-vs-path}.
            \end{proof}

    \subsection{General tree} \label{ssec:general-tree-vs-path}
        We are now ready to prove Theorem \ref{thm:tree-vs-path}, without the constraint on the number of leaves. Our proof strategy is to consider the $m^{1/6}$-core-split $F_1,\dots,F_\ell$ of a tree on $m$ vertices. Then each tree in the split has at most $m^{1/6}$ leaves, so we can use the intermediate result, Theorem \ref{thm:bounded-tree-vs-path}, to find it in the right neighbourhood.

        \begin{proof}[ of Theorem \ref{thm:tree-vs-path}]
            Without loss of generality, we assume that $T$ is out-directed, as otherwise we can look at in-neighbourhoods instead of out-neighbourhoods in $G$. Suppose that $G$, together with a fixed $2$-colouring, has no blue copy of $\dirpath{n}$. Let $c_1$ be the constant from Theorem \ref{thm:bounded-tree-vs-path} for parameters $\eps$ and $2\sigma$. Define $\delta = \frac{\eps^2}{32 \cdot 6}$. We assume that $c \ge \max(2c_1\delta^{-1}, 4\delta^{-2})$.
        
            \begin{claim} \label{claim:core}
                Let $U \subseteq V(G)$ be a set of size at least $\delta N - m $ and let $T'$ be an out-directed tree on at most $m$ vertices with at most $m^{1/6}$ leaves. Then $G[U]$ contains a red copy of $T'$.
            \end{claim}
            \begin{proof}
                Firstly, we claim that $G[U]$ is $(\eps, 2\sigma \log M)$-pseudorandom, where $M = |U|$. Indeed, 
                note that 
                $$
                    M \ge \delta N - m \ge (\delta - 1/c) N \ge \frac{\delta}{2} N \ge \frac{\delta\sqrt{c}}{2} \sqrt{N} \ge \sqrt{N},
                $$
                using $N \ge c$ and $c \ge \frac{4}{\delta^2}$. In particular, $2 \sigma \log M \ge \sigma \log N$, so $G[U]$ is $(\eps, 2\sigma \log M)$-pseudorandom, using  $(\eps,\sigma \log N)$-pseudorandomness of $G$. Next, note that 
                    $$n, m \le \frac{N}{c} \le \frac{2}{\delta c}M \le \frac{M}{c_1},$$ 
                and  
                    $$nm \le \frac{N^2}{c^2 \log N} \le \frac{4}{c^2\delta^2} \frac{M^2}{\log M} \le \frac{M^2}{c_1^2 \log M},$$
                as $c \ge \frac{2c_1}{\delta}$.
                Hence, by Theorem \ref{thm:bounded-tree-vs-path}, $U$ contains either a red $T'$ or a blue $\dirpath{n}$; by assumption it follows that $U$ contains a red $T'$, as required.
            \end{proof}
                
            Let $F_1,\dots, F_\ell$ be the $m^{1/6}$-core-split of $T$. By Proposition \ref{prop:core}, $\ell \le 6$ and each tree in a forest $F_i$ has at most $m^{1/6}$ leaves.
            
            Define $U_0 = V(G)$, and for $i \le 5$ let $U_i$ be the set of vertices in $V(G)$ that have at least $\delta N$ red out-neighbours in $U_{i-1}$.
            
            \begin{claim} \label{claim:size-U-i}
                $|U_i| \ge N/6$ for $i \le 5$.
            \end{claim}
            
            \begin{proof}
                We prove by induction on $i$ that $|U_i| \ge (1 - i/6)N$. This holds trivially for $i = 0$, as $U_0 = V(G)$. Now let $1 \le i \le 5$, and suppose that the statement holds for $i-1$. Consider the set $W := U_{i-1} \setminus U_i$. Suppose that $|U_i| < (1 - i/6)N$, then by induction $|W| \ge N/6$. Also, by the definition of $U_i$, the number of red edges in $W$ is at most $|W|\cdot \delta N \le \frac{\eps^2}{32}|W|^2$ (recall that $\delta = \frac{\eps^2}{32\cdot 6})$. It follows from Lemma \ref{lem:tree-by-number-of-edges} that $W$ contains a blue $\dirpath{n}$, as $\frac{\eps}{4}|W| \ge \max(\sigma \log N, n)$, a contradiction. Hence, $|U_i| \ge (1 - i/6)N \ge N/6$, as required.
            \end{proof}
            
            We now show how to find a red copy of $T$. We first find a red copy of $F_1$ in $U_5$; this is possible due to Claim \ref{claim:core} and the fact that $F_1$ is an out-directed tree on at most $m$ vertices with at most $m^{1/6}$ leaves. Suppose that we found a red copy of $T \setminus (V(F_\ell) \cup \ldots \cup V(F_i))$ for some $2 \le i \le \ell$, such that the vertices corresponding to $F_{i-1}$ are in $U_{7-i}$. We embed the trees in $F_i$ one by one. Let $T'$ be one such tree, and let $u$ be the vertex in $U_{7-i}$ that corresponds to the parent of $T'$ in $T$. Let $W$ be the set of red out-neighbours of $u$ in $U_{6 - i}$ that are still available. By choice of $U_i$, $|W| \ge \delta N - m$, so by Claim \ref{claim:core} there is a red $T'$ in $W$. Continuing this way, we find a copy of $T \setminus (V(F_\ell) \cup \ldots \cup V(F_{i+1}))$ such that the vertices corresponding to $F_i$ are in $U_{6-i}$. Doing this until $\ell = 6$, we find a red copy of $T$. This completes the proof of Theorem \ref{thm:tree-vs-path}.
        \end{proof}

    \section{Tree vs.\ tree} \label{sec:tree-vs-tree}
        In this section we extend Theorem \ref{thm:tree-vs-path} to the case of two general (i.e.\ not necessary directed) trees. We start by proving it for a directed tree with few leaves vs.\ any directed tree (see Theorem \ref{thm:in-vs-out-bounded}); we then remove the assumption that one of the trees has few leaves (Theorem \ref{thm:in-vs-out}); and, finally, we also remove the assumption that the trees are directed (Theorem \ref{thm:tree-vs-tree}). We will often start by embedding a subtree $T'$ of a tree $T$, and then attempt to embed the trees in $T \setminus V(T')$ in the neighbourhood of a suitable vertex in $T'$.
        
        \subsection{Directed tree vs.\ directed tree with few leaves}
            
            Our first goal is to prove the following theorem.
            \begin{thm} \label{thm:in-vs-out-bounded}
                Given $0 < \eps < 1/2$ and $k,\sigma > 0$ there exists a constant $c>0$ such that the following holds. Let $G$ be a tournament on $N$ vertices which is $(\eps,\sigma \log N)$-pseudorandom, let $S$ be a directed tree on $n$ vertices, and let $T$ be a directed trees on $m$ vertices with at most $k$ leaves, where $m, n \le N/c$ and $nm \le \frac{N^2}{c^2\log N}$. Then $G \rightarrow (S, T)$.
            \end{thm}
            
            Before turning to the proof, we give a definition. Let $T$ be an out-directed tree. The \emph{disjoint paths layer} of $T$, denoted $\dpl(T)$, is the collection of paths of $T$ that end at a non-root leaf $u$ and start one vertex after the last branching vertex, or root, between the root and $u$; in the case where $T$ consists of a single vertex (which is the root), we define instead $\dpl(T) = T$. In particular, the vertices in $L(T)$, except for the leaves of $T$, have degree exactly $2$ in $T$.
            
            \begin{figure}[ht]
                \caption{The disjoint paths layer of a tree.}
                \includegraphics{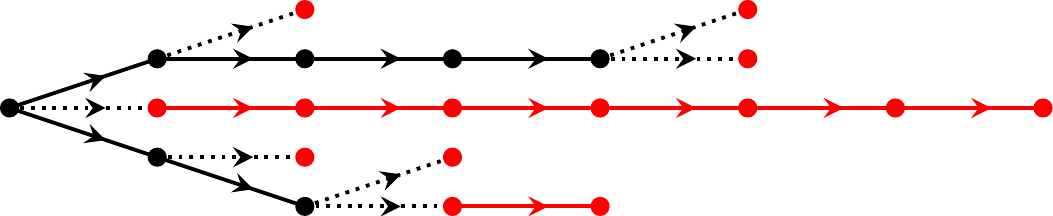}
                \label{fig:dpl}
            \end{figure}

            \begin{prop} \label{prop:path-layer-split}            
                The following properties hold for every out-directed tree $T$.
                \begin{enumerate}[(i)]
                    \item \label{itm:dpl-1}
                    $\dpl(T)$ is a union of pairwise vertex-disjoint directed paths of $T$,
                    \item \label{itm:dpl-2}
                    $T \setminus V(\dpl(T))$ is an out-directed tree,
                    \item \label{itm:dpl-3}
                    the number of non-root leaves in $T \setminus V(\dpl(T))$ is at most half the number of non-root leaves in $T$.
                \end{enumerate}
            \end{prop}
            
            \begin{proof}
                The first two properties are immediate from the definition. Property (\ref{itm:dpl-3}) follows as each non-root leaf in $T \setminus V(\dpl(T))$ sends at least two edges to paths of $\dpl(T)$.
            \end{proof}

            \begin{proof}[ of Theorem \ref{thm:in-vs-out-bounded}]
                Without loss of generality, suppose that $T$ is out-directed. We assume that $G$ has no blue $S$. 
                Let $c_1$ be the constant from Theorem \ref{thm:tree-vs-path} with parameters $\eps$ and $2\sigma$, set $\delta:= \frac{\eps^2}{32(\log k+2)}$, and pick $c$ such that $c \ge \max\{2c_1 / \delta, 4(\log k + 2)/\eps\}$. We use the following claim.
                
                \begin{claim} \label{claim:path-in-U}
                    Let $U$ be a set of at least $\delta N - m$ vertices. Then $U$ contains a red $\dirpath{m}$.
                \end{claim}
                
                \begin{proof}
                    Let $M :=|U|\ge \delta N - m$. Then, using $c \ge 2c_1 / \delta \ge 2 / \delta$,
                    $$
                        M \ge \delta N - \frac{N}{c} \ge \frac{\delta}{2} \cdot N \ge \sqrt{N}.
                    $$
                    Since $G$, and thus $G[U]$, is $(\eps, \sigma \log N)$-pseudorandom, $G[U]$ is $(\eps, 2\sigma \log M)$-pseudorandom. Using $c \ge 2 c_1 / \delta$, we have $M \ge (\delta / 2) N \ge (c_1 / c) N$.
                    Thus, by the assumptions on $n$ and $m$,
                    $$
                        n,m \le \frac{N}{c} \le \frac{M}{c_1} \qquad \text{and} \qquad nm \le \frac{N^2}{c^2 \log N} \le \frac{M^2}{c_1^2 \log M}.
                    $$
                    Hence, by definition of $c_1$ (according to Theorem \ref{thm:tree-vs-path}), $U$ contains a red $\dirpath{m}$ or a blue $S$. Since we assumed that the latter does not hold, $U$ contains a red $\dirpath{m}$, as required.
                \end{proof}
                
                Our plan is to embed a red copy of $T$ layer by layer. To this end, define $T_0 := T$ and, for $i \ge 1$, $T_i := T_{i-1} \setminus V(\dpl(T))$, and let $h$ be the largest $i$ such that $T_i$ is non-empty. Note that $T_h$ is a singleton (as the root is not removed unless the root is the only vertex), and, by Proposition \ref{prop:path-layer-split} (\ref{itm:dpl-3}), $T_i$ has at most $k \cdot 2^{-i}$ non-root leaves; in particular, $h \le \log k + 1$.
                
                Define $U_0 := V(G)$, and for $1 \le i \le h$ let $U_i$ be the set of vertices in $U_{i-1}$ whose red out-degree in $U_{i-1}$ is at least $\delta N$. We shall need the following claim.

                \begin{claim}\label{claim:high-deg-embed}
                    $U_h \neq \emptyset$.
                \end{claim}

                \begin{proof}
                    The proof is essentially identical to that of Claim \ref{claim:size-U-i}.
                    We prove by induction that $|U_i| \ge (1 - \frac{i}{h+1})N$ for $0 \le i \le h$. This is trivial for $i = 0$, as $U_0 = V(G)$. Let $0 < i \le h$, and suppose that the statement holds for $i-1$, i.e.\ $|U_{i-1}| \ge (1 - \frac{i-1}{h+1}) N$. Set $W := U_{i-1} \setminus U_i$. Suppose that $|U_i| < (1 - \frac{i}{h+1})N$; so $|W| \ge N / (h+1)$. We now wish to apply Lemma \ref{lem:tree-by-number-of-edges}. To do so, note that, by definition of $U_i$, the number of red edges in $W$ is at most $|W| \cdot \delta N = |W| \cdot \frac{\eps^2}{32 (\log k + 2)} \cdot N ֿ\le \frac{\eps^2}{32}|W| \cdot \frac{N}{h+1} \le  \frac{\eps^2}{32}|W|^2$, using the defintion of $\delta$ and the bounds $h \le \log k + 1$ and $|W| \ge N / (h+1)$. We also have $\frac{\eps}{4} |W| \ge \sigma \log N$ (since $N\ge c$ and we take $c$ large enough, in terms of $\sigma,\eps,k$). Thus, by Lemma \ref{lem:tree-by-number-of-edges}, $G$ contains any blue tree on at most $\frac{\eps}{4}|W|$ vertices. Since $m \le N / c \le \frac{\eps}{4} \cdot N / (h+1) \le \frac{\eps}{4} |W|$ (using $c \ge 4 (\log k + 2) / \eps \ge 4 (h+1) / \eps$), it follows that $G$ contains a blue copy of $T$, a contradiction.
                \end{proof}
                We now show that there is a red copy of $T_i$ in $U_i$, by induction on $0 \le i \le h$. Since $T_h$ is a singleton and $U_h$ is non-empty, there is indeed a red copy of $T_h$ in $U_h$. Now suppose that for some $0 \le i < h$, there is a red copy of $T_{i+1}$ in $U_{i+1}$. Recall that $T_{i+1} = T_i \setminus V(\dpl(T_i))$, hence it suffices to show that the paths in $\dpl(T_i)$ can be embedded in the red out-neighbourhoods of the corresponding vertices in $T_{i+1}$. We embed the paths in $\dpl(T_i)$ one by one. Let $P$ be a path in $\dpl(T_i)$ of order $\ell$, let $v$ be its start-vertex and let $u$ be the vertex in $U_{i+1}$ that corresponds to the parent of $v$ in $T$. Denote by $W$ the red out-neighbours of $u$ in $U_i$ which are still available. Then, since $u$ is in $U_{i+1}$ and at most $m$ vertices are used, $|W| \ge \delta N - m$. By Claim \ref{claim:path-in-U}, $W$ contains a red $P$, as required. We are thus able to embed each of the paths in $\dpl(T_i)$ in $U_i$ so as to obtain a red copy of $T_i$ in $U_i$. In particular, by taking $i = 0$, we see that $G$ has a red copy of $T$, as required for the proof of Theorem \ref{thm:in-vs-out-bounded}.
            \end{proof}
            
        \subsection{Directed trees}

            With the next theorem we further generalise the result to the case of any directed trees $S$ and $T$. We once again obtain a reduction to the previous result, Theorem \ref{thm:in-vs-out-bounded}. This time we make use of $k$-cores, which we already encountered in the proof of Theorem \ref{thm:tree-vs-path} (see Definition \ref{def:core}).

            \begin{thm}\label{thm:in-vs-out}
                Given $0 < \eps < 1/2$ and $\sigma > 0$, there exists a constant $c>0$ such that the following holds. Let $G$ be a tournament on $N$ vertices which is $(\eps,\sigma \log N)$-pseudorandom. Then $G \rightarrow (S,T)$ for any directed trees $S$ and $T$ on $n$ and $m$ vertices, respectively, where $n, m \le N/c$ and $nm \le \frac{N^2}{c^2\log N}$.
            \end{thm}

            \begin{proof}
                Our goal is to reduce the statement of this theorem to the case when one of the trees has a constant number of leaves. We iteratively make the trees $S$ and $T$ smaller, using Theorem \ref{thm:in-vs-out-bounded}, until one of them becomes a singleton. 
    
                Define $\delta = \eps^2 / 64$, $\ell :=  8/\delta^2$, $k := \ell^2$, let $c_1$ be the constant from Theorem \ref{thm:in-vs-out-bounded} with parameters $\eps$, $3\sigma$ and $k$, and let $c := \max\{c_1 \ell, 8 \ell / \eps\}$. Set $h := \ceil{\log_k N}$ for $0 \le i \le h$, and write $n_i := n \cdot k^{-i}$, $m_i := m \cdot k^{-i}$,  and $N_i := N \cdot \ell^{-i}$. We shall use the following proposition.
                \begin{prop} \label{prop:dir-tree-tree-easy-facts}
                    The following properties hold.
                    \begin{enumerate}[(i)]
                        \item \label{itm:fact-arrow}
                            Let $U$ be a set of at least $N_{i+1}$ vertices, let $S$ be a directed tree on $n_i$ vertices, and let $T$ be a directed tree on $m_i$ vertices with at most $k$ leaves. Then $U$ contains a blue $S$ or a red $T$.
                        \item \label{itm:fact-large-deg}
                            Let $U$ be a set of at least $N_{i+1}$ vertices. Then either it contains a blue copy of any tree on $n_i$ vertices, or the set of vertices in $U$ whose red out-degree in $U$ is at least $\delta |U|$ has size at least $|U|/2$.
                    \end{enumerate}
                \end{prop}
                
                \begin{proof}
                    Firstly, note that for every $0 \le i \le h$
                    \[
                        N_i 
                        \ge N \cdot \ell^{-\log_k N - 1} 
                        = N \cdot k^{-\frac{1}{2}\log_k N} \cdot \frac{1}{\ell}  
                        = \frac{\sqrt{N}}{\ell} \ge N^{1/3}.
                    \]
                    It follows that every subset $U \subseteq V(G)$ of size at least $N_i$, where $0 \le i \le h$, is $(\eps, 3\sigma \log |U|)$-pseudorandom.
                    
                    Note that $n_i = n k^{-i} \le n \ell^{-i} \le (N/c) \ell^{-i} = N_{i+1} \ell / c \le N_{i+1} / c_1$ (using $c \ge c_1 \ell$). Similarly, $m_i \le N_{i+1} / c_1$ and $n_i m_i \le N_{i+1}^2 / c_1^2 \log N_{i+1}$. Property (\ref{itm:fact-arrow}) thus follows from the definition of $c_1$ (via Theorem \ref{thm:in-vs-out-bounded}).
                    
                    Property (\ref{itm:fact-large-deg}) can be deduced from Lemma \ref{lem:tree-by-number-of-edges} as follows. Suppose that the set $X$ of vertices in $U$ whose red out-degree is smaller than $\delta |U|$ has size at least $|U|/2$. Then the number of red edges spanned by $X$ is at most $|X|\cdot \delta |U| \le \frac{\eps^2}{32}|X|^2$. Thus, by Lemma \ref{lem:tree-by-number-of-edges}, $G[X]$ contains a blue copy of any tree on at most $\frac{\eps}{4}|X| \ge n_i$ vertices, as required, where we used the inequalities $\frac{\eps}{4}|X| \ge \frac{\eps}{8}N_{i+1}\ge \frac{\eps c}{8 \ell} n_i \ge n_i$ (using $n_i \le N_{i+1} \ell / c$ and $c \ge 8 \ell / \eps$) and $\frac{\eps}{4}|X| \ge \frac{\eps}{8}N_{i+1} \ge \sigma \log N$. 
                \end{proof}
                
                We complete the proof with the following claim.
                
                \begin{claim} \label{claim:ind-tree-tree}
                    Let $U \subseteq V(G)$ be a set of size at least $N_i$, where $0 \le i \le h$, and let $S$ and $T$ be directed trees of order $n_i$ and $m_i$, respectively. Then $U$ contains a blue $S$ or a red $T$.
                \end{claim}
                
                \begin{proof}
                    We prove the claim by induction on $i$. Note that when $i = h$ the claim holds trivially as $n_h, m_h \le 1$ and $N_h \ge 1$. Now suppose that $0 \le i < h$ and the claim holds for $i+1$.
                    
                    Suppose that $U$ does not contain a blue $S$ or a red $T$. For convenience, we assume that $S$ and $T$ are out-directed; the remaining cases follow similarly. Let $S'$ and $T'$ be the $k$-cores of $S$ and $T$, respectively. Then $S'$ and $T'$ have at most $k$ leaves, $S \setminus V(S')$ is a forest of trees of order at most $n_{i+1}$, and $T \setminus V(T')$ is a forest of trees of order at most $m_{i+1}$. 
                    
                    Let $X$ be the set of vertices in $U$ whose red out-degree in $U$ is at least $\delta |U|$. Then, by Proposition \ref{prop:dir-tree-tree-easy-facts} (\ref{itm:fact-large-deg}) and the assumption that $U$ does not contain a blue $S$, we have $|X| \ge |U|/2 \ge N_i / 2 \ge N_{i+1}$.
                    By Proposition \ref{prop:dir-tree-tree-easy-facts} (\ref{itm:fact-arrow}) and the assumption that $U$ does not have a blue $S$, $X$ contains a red $T'$. We attempt to extend the copy of $T'$ to a red $T$ in $U$ by attaching, one at a time, copies of the trees in $T \setminus V(T')$. As $U$ does not have a red copy of $T$, at some point we fail. Let $T''$ be the tree in $T \setminus V(T')$ that we fail to embed (while $T'$ and some of $T \setminus V(T')$ is already embedded). Denote the root of $T''$ by $u$, and let $u'$ be the vertex in $X$ in which we embedded the parent of $u$ in $T$.
                    
                    Denote by $Y$ the set of red out-neighbours of $u'$ in $U$ which have not been used yet, so by the failure to embed $T''$, $Y$ does not have a red $T''$. Let $Y'$ be the set of vertices in $Y$ whose blue out-degree in $Y$ is at least $\delta |Y|$. Then, $|Y| \ge \delta|U| - m_i \ge \frac{\delta}{2}|U| \ge \frac{\delta}{4}N_i \ge N_{i+1}$ hence, by Proposition \ref{prop:dir-tree-tree-easy-facts} (\ref{itm:fact-large-deg}), with red and blue swapped, $|Y'| \ge |Y|/2 \ge \frac{\delta}{8}N_i \ge  N_{i+1}$. As $Y$, and thus $Y'$, does not contain a red $T''$, it follows from Proposition \ref{prop:dir-tree-tree-easy-facts} (\ref{itm:fact-arrow}) that $Y'$ contains a blue $S'$. Again, we try to extend this copy of $S'$ to a blue copy of $S$ in $Y'$, by attaching one tree of $S \setminus V(S')$ at a time. As there is no blue copy of $S$ in $Y'$, at some point we fail; denote by $S''$ the tree that we fail to embed. Let $v$ be the root of $S''$, and let $v'$ be the vertex in $Y'$ where we embedded the parent of $v$ in $S$.
                    
                    Let $Z$ be the set of blue out-neighbours of $v'$ in $Y'$ which are not used. Then $|Z| \ge \delta|Y| - n_i \ge \frac{\delta}{2}|Y| \ge \frac{\delta^2}{8}N_i = N_{i+1}$ and $Z$ does not have a red $T''$ or a blue $S''$, contrary to the induction hypothesis. It follows that $U$ contains a red $T$ or a blue $S$, as required.
                \end{proof}
    
                The proof of Theorem \ref{thm:in-vs-out} follows immediately from Claim \ref{claim:ind-tree-tree} by taking $i=0$.
            \end{proof}

        \subsection{General trees}
            Our final aim is to generalise Theorem \ref{thm:in-vs-out} to arbitrary oriented trees, as follows.
            \begin{thm} \label{thm:tree-vs-tree}
                Given $0 < \eps < 1/2$ and $\sigma > 0$, there exists a constant $c>0$ such that the following holds. Let $G$ be a tournament on $N$ vertices which is $(\eps,\sigma \log N)$-pseudorandom, and let $S$ and $T$ be trees of orders $n$ and $m$, respectively, where $m, n \le N/c$ and $nm \le \frac{N^2}{c^2 \log N}$. Then $G \rightarrow (S,T)$.
            \end{thm}
            
            We will use the next definition and lemma in the proof.
    
            \begin{defn} \label{def:mindeg}
                Let $G$ be an oriented graph and $k$ a positive constant. We call a pair of disjoint subsets $(A,B)\subseteq V(G)^2$ a \emph{$k$-mindegree pair} if every vertex in $A$ has at least $k$ out-neighbours in $B$ and every vertex in $B$ has at least $k$ in-neighbours in $A$.
            \end{defn}
    
            \begin{lem} \label{lem:mindeg}
                Let $0 < \delta < \frac{1}{4}$. In every oriented graph $G$ with at least $\delta |G|^2$ edges, there is a $\frac{\delta}{4} |G|$-mindegree pair.
            \end{lem}

            \begin{proof}
                Let us define a partition $(X,Y)$ of $V(G)$ by putting each vertex independently with probability $1/2$ either in $X$ or in $Y$. Note that the expectation of $e(X,Y)$ is at least $e(G)/4\ge \frac{\delta}{4}|G|^2$. Thus, there exist disjoint sets $X$ and $Y$ with $e(X,Y)\ge \frac{\delta}{4}|G|^2$.

                Now we consider the underlying subgraph of $G$ whose edges are those going from $X$ to $Y$. We remove one by one all vertices with degree less than $\frac{\delta}{4}|G|$ in this underlying graph. Let $A\subseteq X$ and $B\subseteq Y$ be the sets of remaining vertices. Note that both $A$ and $B$ are non-empty, since otherwise all vertices would be removed by this process, each contributing less than $\frac{\delta}{4}|G|$ edges. This would imply $e(X,Y)<\frac{\delta}{4}|G|^2$, a contradiction. Therefore, $(A,B)$ is a $\frac{\delta}{4}|G|$-mindegree pair in $G$.
            \end{proof}

            We shall use Theorem \ref{thm:in-vs-out} in our proof of Theorem \ref{thm:tree-vs-tree}. For this we need a suitable split of $T$.
            
            Let $v$ be the root of $T$. Let $F_1$ be the induced subtree of $T$ containing $v$ and all vertices of $T,$ which can be reached from $v$ by following in-edges (it is possible that $F_1$ contains only $v$). Let $U_2$ be the set of roots of the trees in the forest $T\setminus V(F_1)$. We define $F_2$ to be the forest of induced subtrees of $T$ consisting of the vertices in $U_2$ and all vertices in $T\setminus V(F_1)$ that can be reached from $U_2$ by following out-edges. We continue this procedure and eventually we obtain a split of $T$ into layers of in and out-forests $F_1, \dots, F_\ell,$ such that the forest $F_i$ consists of in-directed trees for odd $i\in[\ell]$ and of out-directed trees for even $i\in [\ell]$. Moreover, all edges in $T$ are either contained in a forest $F_i$ or are between consecutive layers $F_i$ and $F_{i+1}$, and they are directed from $F_i$ to $F_{i+1}$ if $i$ is odd, and are directed from $F_{i+1}$ to $F_i$ if $i$ is even. We call this split the \emph{in-out split} of $T$.
            
            \begin{figure}[ht]
                \caption{The in-out split of a tree.}
                \includegraphics{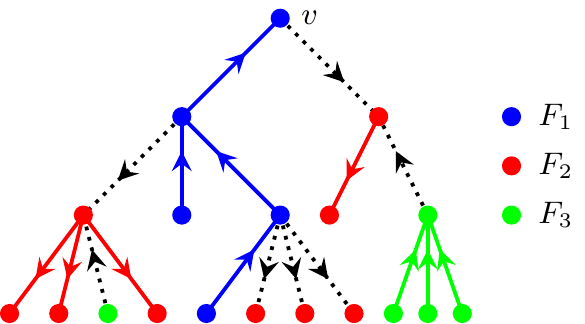}
                \label{fig:out-in-split}
            \end{figure}

            \begin{proof}[ of Theorem \ref{thm:tree-vs-tree}]
                We first prove the theorem under the additional assumption that $S$ is directed, using Theorem \ref{thm:in-vs-out}, and then we use this to prove the theorem in full generality. In order to avoid repeating the arguments, we use the Proposition \ref{prop:dir-to-undir} below.
                
                Let $\delta = \frac{\eps^2}{32}$. Let $c_1$ be the constant from Theorem \ref{thm:in-vs-out} with parameters $\eps$ and $2 \sigma$. Without loss of generality $c_1 \ge 8/\delta$, and let $c = c_1^3$.
                Let $N,n,m$ be fixed (such that the inequalities in the statement of the theorem hold), and let $G$ be a $2$-coloured tournament on $N$ vertices. 
                
                \begin{prop} \label{prop:dir-to-undir}
                    Let $N/c_1 \le M \le N$. Let $U \subseteq V(G)$ be a set of size $M$, and suppose that every subset of $U$ of size at least $M/c_1$ contains a red copy of every \emph{directed} tree of order at most $m$. Then $U$ contains a blue copy of every tree (not necessarily directed) of order $n$, or a red copy of every tree of order $m$. 
                    
                    The same holds with the roles of red and blue, and the roles of $n$ and $m$, swapped.
                \end{prop}
                
                \begin{proof}
                    As we have done already several times, by Lemma \ref{lem:tree-by-number-of-edges} we can assume that $U$ spans at least $\delta M^2$ red edges, since otherwise $U$ contains a blue copy of every tree of order $(\eps/4)M \ge n$. Then by Lemma \ref{lem:mindeg} there exist disjoint sets $A, B \subseteq U$ such that $(A,B)$ is a $(\delta / 4)M$-mindegree pair in the red subgraph of $G$. Let $T$ be a tree on $m$ vertices, consider its in-out split $F_1,\dots,F_\ell$, and denote by $V_1,\dots, V_\ell$ the corresponding partition of vertices of the tree $T$; recall that $F_1$ is an in-directed subtree of $T$. We will embed every in-directed tree of the in-out split inside $A$ and every out-directed tree in $B$.
    
                    \begin{claim} \label{claim:embedding}
                        For every $i\in[\ell]$ there is a red copy of $T[V_1\cup \dots \cup V_i]$ such that $V_i$ is embedded in $A$ if $i$ is odd and in $B$ if $i$ is even.
                    \end{claim}
                    \begin{proof}
                        We prove this by induction. For the basis, note that $F_1$ is a single in-directed tree and $|A| \ge (\delta / 4) M \ge M / c_1$, thus by assumption there is a red copy of $F_1$ inside $A$.
            
                        Now let us assume that the claim holds for $1 \le i-1 < \ell$. For convenience we assume that $i$ is even; the case where $i$ is odd follows similarly. So, we have found a red copy of $T[V_1\cup \dots \cup V_{i-1}]$ such that $V_{i-1}$ is embedded in $A$. Now we need to show how to embed the trees of the forest $F_i$. Let $T'$ be one of the trees in the forest $F_i$ and $v \in A$ be the vertex corresponding to the parent of the roof of $T'$ in $T[V_1\cup \dots \cup V_{i-1}]$. Since $(A,B)$ is a $(\delta / 4) M$-mindegree pair, $v$ has at least $(\delta / 4) M$ red out-neigbours in $B$. So far we embedded at most $n$ vertices of the tree $T$, so the number of available vertices in the neighbourhood is at least $(\delta / 4) M-n \ge (\delta / 4) M - M/c_1 \ge M/c_1$. Therefore, by assumption, there is a red copy of $T'$ in $B$ rooted at some vertex $w$, such that edge $vw$ is red.
                        
                        This way we can embed all the trees in $F_i$ and extend the red copy of $T[V_1\cup \dots \cup V_{i-1}]$ to a red copy of $T[V_1\cup \dots \cup V_{i}]$ satifying the conditions of the claim. 
                    \end{proof}
        
                    By Claim \ref{claim:embedding} with $i=\ell$, $U$ contains a red $T$. As $T$ was an arbitrary tree on $m$ vertices, the proof is complete. An analogous argument can be used to prove the statement of the proposition with the roles of red and blue, and of $m$ and $n$, swapped.
                \end{proof}
                We now show how to complete the proof of Theorem \ref{thm:tree-vs-tree} using Proposition \ref{prop:dir-to-undir}. Suppose that there exists a subset $U \subseteq V(G)$ of size at least $N/c_1$, whose subsets of size at least $|U|/c_1$ all contain a red copy of every directed tree of order $m$. Then, by Proposition \ref{prop:dir-to-undir}, $U$ contains a red copy of every tree of order $m$ or a blue copy of every tree of order $n$, and we are done. Thus we may assume that every subset $U \subseteq V(G)$ has a subset $W_U$ of size at least $|U|/c_1$ such that $W_U$ does not contain a red $T_U$, for some directed tree $T_U$ of order $m$. But then, by Theorem \ref{thm:in-vs-out}, every such $W_U$ contains a blue copy of every directed tree on $n$ vertices (using the definition of $c_1$, and the inequalities $n,m \le N/c = N/c_1^3 \le |W_U|/c_1$ and $nm \le N^2/c^2 \log N \le |W_U|^2/c_1^2 \log |W_U|$). In particular, every set $U \subseteq V (G)$ of size at least $N/c_1$ contains a blue copy of every directed tree on $n$ vertices. By Proposition \ref{prop:dir-to-undir} again (with the roles of red and blue and $n$ and $m$ swapped), either $G$ contains a red copy of every tree on $m$ vertices, or a blue copy of every tree on $n$ vertices, as required.
            \end{proof}
    
    \section{Concluding remarks and open problems}
    
        In this paper we have proved that, \whp, in every $2$-edge-colouring of a random tournament on $Cn\sqrt{\log n}$ vertices there exists a monochromatic copy of any tree of order $n$.
    
        Buci\'c, Letzter and Sudakov \cite{complete-directed-ramsey} proved tight results for both oriented and directed Ramsey numbers of trees for the case of more than two colours as well. It seems that the methods used in their proofs do not extend directly to the random tournament setting, so it could be very interesting to extend our result to $k$-colours. In the case of paths they showed in \cite{path-vs-path} that, \whp, in any $k$-edge colouring of a random tournament on $\Omega(n^{k-1} \sqrt{\log n})$ vertices, there is a monochromatic path of length $n$. Moreover, an example by Ben-Eliezer, Krivelevich and Sudakov \cite{ben2012size} shows that there is a $k$-edge colouring of any tournament on $cn^{k-1}(\log n) ^{1/k}$ vertices with no monochromatic paths of length $n$, for some constant $c > 0$. We believe the upper bound should be tight, for random tournaments, but the $k$-colour case is still open, even for directed paths.
    
        Burr and Erd\H{o}s \cite{burr73} initiated the study of Ramsey numbers of bounded degree graphs in $1975$. They conjectured that the Ramsey number of bounded degree graphs is linear in their size. This was subsequently proved by Chv\'atal, R\"odl, Szemer\'edi and Trotter \cite{chvatal83}. The dependence of the constant factor on the maximum degree in this bound was later improved, first by Eaton \cite{eaton98}, then by Graham, R\"odl and Ruci\'nski \cite{graham00} and the currently best bound is due to Conlon, Fox and Sudakov \cite{conlon12}. Buci\'c, Letzter and Sudakov \cite{complete-directed-ramsey} pose an interesting analogous problem in the oriented and directed Ramsey settings. They ask if for every $d$ there is a constant $c = c(d)$ such that any tournament on $cn$ vertices contains any acyclic graph on at most $n$ vertices with maximum degree at most $d$. This can be thought of as the one colour version of the more general question of determining the $r$-colour oriented Ramsey number of bounded degree acyclic graphs. A similar question arises naturally in the random setting. Here the one colour version is a simple consequence of the directed version of the Regularity Lemma of Szemer\'edi \cite{sze1978-regularity} due to Alon and Shappira \cite{directed-regularity}. However, the question of the two colours is open and interesting and it seems likely that a result in any setting could also help with the other settings.
    
        Theorem \ref{thm:main} is tight up to a constant factor, as long as the only information we are given on the tree is its order. However, it is not tight for every tree of order $n$. For example, if the tree in question $T$ is a star of order $n$, then it is not hard to see that the random tournament $G$ is only required to have order $\Omega(n)$ in order to satisfy $G \to T$, as opposed to a bound of $\Omega(n \sqrt{\log n})$ which is needed for a directed path on $n$ vertices, or for trees which contain directed subpaths of order $\Omega(n)$. With this in mind, it is natural to ask if the tight bound for a tree $T$ depends only on the order of the tree and the length of its longest directed subpath, denoted by $\ell(T)$. More precisely, Buci\'c, Letzter and Sudakov \cite{complete-directed-ramsey} ask if the directed Ramsey number of a tree is $O(|T| \cdot \ell(T))$; if this holds, it can readily be seen to be tight. They prove that this holds for oriented paths. It would also be interesting to tackle this question in the random tournaments setting. 

\subsection*{Acknowledgements}
        We would like to thank the anonymous referee for helpful comments.
    
\providecommand{\bysame}{\leavevmode\hbox to3em{\hrulefill}\thinspace}
\providecommand{\MR}{\relax\ifhmode\unskip\space\fi MR }
\providecommand{\MRhref}[2]{%
  \href{http://www.ams.org/mathscinet-getitem?mr=#1}{#2}
}
\providecommand{\href}[2]{#2}

\end{document}